\documentclass[10pt]{article}

\usepackage{amssymb}
\usepackage{amsmath}
\usepackage{amsfonts}
\usepackage{geometry}

\usepackage[active]{srcltx}

\usepackage{amsfonts}
\usepackage{amsmath}
\usepackage{amssymb}
\usepackage{amsmath,amsfonts,amsthm,amscd,amssymb}
\usepackage[colorlinks=true, pdfstartview=FitV, linkcolor=blue, citecolor=blue, urlcolor=blue]{hyperref}
\usepackage{authblk}

\usepackage{enumitem}

\newtheorem{theorem}{Theorem}

\newtheorem{condition}[theorem]{Condition}

\newtheorem{corollary}[theorem]{Corollary}

\newtheorem{definition}[theorem]{Definition}

\newtheorem{lemma}[theorem]{Lemma}

\newtheorem{proposition}[theorem]{Proposition}
\newtheorem{remark}[theorem]{Remark}

\usepackage{soul,xcolor}

\newcommand{\mr}{\mathbb{R}}
\newcommand{\E}{\mathbb{E}}
\newcommand{\mc}{\mathcal}

\newcommand{\lan}{\langle}
\newcommand{\ran}{\rangle}

\newcommand{\td}{\tilde}
\newcommand{\eps}{\epsilon}

\newcommand{\rmd}{\mathrm{d}}

\title{A McKean-Vlasov SDE and particle system with interaction from reflecting boundaries}
\author[1]{Michele Coghi \thanks{michele.coghi@unitn.it}}
\author[2]{Wolfgang Dreyer \thanks{dreyer@wias-berlin.de}}
\author[3,2]{Peter K.~Friz \thanks{friz@math.tu-berlin.de }}
\author[2]{Paul Gajewski\thanks{gajewskp@wias-berlin.de}}
\author[2]{Clemens Guhlke\thanks{guhlke@wias-berlin.de}}
\author[4]{Mario Maurelli \thanks{mario.maurelli@unimi.it }}
\affil[1]{Dipartimento di Matematica, Universit\`a degli Studi di Trento, via Sommarive 14, 38123 Povo (Trento), Italy}
\affil[2]{Weierstrass Institute for Applied Analysis and Stochastics, Mohrenstra\ss e 39, 10117 Berlin, Germany}
\affil[3]{Institut f\"ur Mathematik, Technische Universit\"at Berlin, Stra\ss e des 17. Juni 136, 10623 Berlin,	Germany}
\affil[4]{Dipartimento di Matematica, Universit\`a degli Studi di Milano, via Saldini 50, 20133 Milano, Italy}

\begin{document}



\maketitle

\begin{abstract}
We consider a one-dimensional McKean-Vlasov SDE on a domain and the associated mean-field interacting particle system. The peculiarity of this system is the combination of the interaction, which keeps the average position prescribed, and the reflection at the boundaries; these two factors make the effect of reflection non local. We show pathwise well-posedness for the McKean-Vlasov SDE and convergence for the particle system in the limit of large particle number.
\end{abstract}


%

\tableofcontents

\section{Introduction}

In this paper, we consider a system of $N$ interacting one-dimensional diffusions, with the following two main features: a) they are confined in a bounded domain with reflecting boundaries; b) their empirical average is prescribed. We show that, as the number of particles $N$ goes to infinity, the system converges to the unique solution to a suitable McKean-Vlasov SDEs on the domain.

\textbf{The model and the results}. The propotypical example is the following one:
\begin{align}
\begin{aligned}\label{eq:prototype}
&\rmd X^{i,N}_t = \rmd W^i_t + \rmd K^N_t -\rmd k^{i,N}_t,\\
&X^{i,N}_t\in [0,1],\quad \rmd k^{i,N}_t = n(X^{i,N}_t)\rmd |k^{i,N}|_t,\quad \rmd |k^{i,N}|_t = 1_{X^{i,N}_t\in \{0,1\}}\rmd |k^{i,N}|_t,\\
&\frac{1}{N}\sum_{i=1}^N X^{i,N}_t =q.
\end{aligned}
\end{align}
Here $W^i$ are independent real Brownian motions, $q$ is the given average in $(0,1)$, $n$ is the outer normal on $\partial[0,1] = \{0,1\}$ and the solution is a triple $X^{(N)}=(X^{i,N})_{i=1,\ldots N}$, $k^{(N)}=(k^{i,N})_{i=1,\ldots N}$, $K^N$ satisfying the above system; $|k^{i,N}|$ denotes the total variation process of $k^{i,N}$. We will sometimes omit the superscript $N$ from the notation. The term $-k^i$ represents the reflection of the process $X^i$ at the boundary of $[0,1]$ and the term $K$ (independent of $i$) represents the interaction between the particles, which keeps the average equal to $q$.
More generally, for modelling purpose, we consider also the case with a given drift $\mu:[0,1]\rightarrow \mr$, a given time-dependent average $q:[0,T]\rightarrow (0,1)$ and a constant noise intensity $\sigma\in \mr$, namely we take the system
\begin{align}
\begin{aligned}\label{eq:particle_eq_intro}
&\rmd X^{i,N}_t = -\mu(X^{i,N}_t)\rmd t +\sigma\rmd W^i_t + \rmd K^N_t -\rmd k^{i,N}_t,\\
&X^{i,N}_t\in [0,1],\quad \rmd k^{i,N}_t = n(X^{i,N}_t)\rmd |k^{i,N}|_t,\quad \rmd |k^{i,N}|_t = 1_{X^{i,N}_t\in \{0,1\}}\rmd |k^{i,N}|_t,\\
&\frac{1}{N}\sum_{i=1}^N X^{i,N}_t =q(t).
\end{aligned}
\end{align}
The last line of the above system can be easily converted into an expression for $K^N$ in terms of $X^{(N)}$ and $k^{(N)}$, namely
\begin{align}\label{eq:K_k}
\rmd K^N_t = (\frac{1}{N}\sum_{i=1}^N \mu(X^i_t) +\dot{q}(t))\rmd t -\sigma\frac{1}{N}\sum_{i=1}^N\rmd W^i_t +\frac{1}{N}\sum_{i=1}^N \rmd k^i_t.
\end{align}
The main novelty of this work is the peculiar combination of the reflecting boundary and the condition on the average of the particles. This combination is reflected in formula \eqref{eq:K_k}, where the interaction $\rmd K$ depends also on the empirical average of $\rmd k^i$. To guess the limiting behaviour (as $N\to \infty$) of the system \eqref{eq:particle_eq_intro}, we can replace the average over particles $N^{-1}\sum_{i=1}^N$ with the average over the probability space $\E$. In this way, we get the following McKean-Vlasov SDE on the domain $[0,1]$:
\begin{align}
\begin{aligned}\label{eq:McKVla_intro}
&\rmd \bar X_t = -\mu(\bar X_t)\rmd t +\sigma\rmd W_t + \rmd \bar K_t -\rmd \bar k_t,\\
&\bar X_t\in [0,1],\quad \rmd \bar k_t = n(\bar X_t)\rmd |\bar k|_t,\quad \rmd |\bar k|_t = 1_{\bar X_t\in \{0,1\}}\rmd |\bar k|_t,\\
&\E \bar X_t =q(t),
\end{aligned}
\end{align}
where $W$ is a real Brownian motion and the solution is a triple $\bar{X},\bar{k},\bar{K}$ satisfying the above SDE. As in the particle system, the last line of \eqref{eq:McKVla_intro} can be converted into an expression for $\bar K$:
\begin{align}\label{eq:K_k_bar}
\rmd \bar K_t = (\E \mu(\bar X_t) +\dot{q}(t))\rmd t +\E \rmd \bar k_t.
\end{align}
Our main results Theorems \ref{thm:one}, \ref{thm:convergence} and Proposition \ref{pro: convergence rate} state roughly speaking that the McKean-Vlasov SDE \eqref{eq:McKVla_intro} is well-posed in the pathwise sense and the empirical measure $\frac{1}{N}\sum_{i=1}^N \delta_{X^{i,N}}$ from the system \eqref{eq:particle_eq_intro} converges in probability, as $N\to\infty$, to the law of the unique solution to \eqref{eq:McKVla_intro}, with its time marginals converging in $L^1$.

\textbf{Motivation}. Our motivation to study this system comes from a specific model for charging and discharging in a lithium-ion battery, introduced in \cite{DreGulHer2011} and further studied and expanded for example in \cite{DHMRW2015,GGMFD2018}. In this model, roughly speaking, the lithium atoms enter and exit iron phosphate particles in the cathode. The $Y^i_t$ represents the filling degree of the $i$-th iron phosphate particle at time $t$ (for example, $Y^i=1$, $=0$ resp., stands for the $i$-th particle fully filled with lithium atoms, fully empty resp.); by this definition of $Y^i$, $Y^i$ has to stay in $[0,1]$. The prescribed average $q(t)$ of $Y^i_t$ represents the current in the battery, which is given and is proportional to the percentage of lithium atoms inside the ensemble of particles. The reason to consider reflecting boundaries $Y^i\in \{0,1\}$ comes from the boundary conditions in the Fokker-Planck equation in \cite{DHMRW2015}; this choice of boundary conditions is convenient mathematically, though the physical motivation is less clear. From a mathematical perspective, \cite{DHMRW2015} shows global well-posedness for the nonlinear nonlocal Fokker-Planck equation associated with the McKean-Vlasov SDE \eqref{eq:McKVla_intro}, namely
\begin{align}
\begin{aligned}\label{eq:FPE_intro}
&\partial_t u(t,x) +\partial_x [(-\mu(x)+\dot{\bar K}_t)u(t,x)] =\frac{\sigma^2}{2}\partial_x^2 u(t,x),\quad t>0,\,x\in (0,1),\\
&\frac{\sigma^2}{2}\partial_x u(t,x) +(\mu(x)-\dot{\bar K}_t)u(t,x) = 0,\quad t>0,\,x\in\partial(0,1),\\
&\int_0^1 xu(t,x) \rmd x = q(t),\quad t\ge 0.
\end{aligned}
\end{align}
The paper \cite{GGMFD2018} considers the particle system \eqref{eq:FPE_intro}, associated with \eqref{eq:particle_eq_intro}, even in a more general version (to take into account variations in the radius of iron phosphate particles), but removes the boundaries: without boundaries, the particle system \eqref{eq:particle_eq_intro} is reduced to a classical system of mean field interacting diffusions, for which convergence to the corresponding McKean-Vlasov SDE is well-known. Hence the current paper arises from the natural (from a mathematical viewpoint) question whether convergence of the particle system for the model \eqref{eq:FPE_intro} holds. We also point out that interacting diffusions with constraints both on the domain  and on the empirical measure of the diffusions appear in several contexts, see e.g. \cite{BriChaGuiLab2020,Jab2017,Bar2020} below.

\textbf{Background}. McKean-Vlasov SDEs are SDEs where the drift depends also on the law of the solution, namely SDEs of the form
\begin{align}\label{eq:McKVla_no_bdry}
\rmd \bar X_t = b(\bar X_t,\text{Law}(\bar X_t))\rmd t +\rmd W_t,
\end{align}
where $W$ is a given Brownian motion (we do not consider here the case of general diffusion coefficients). McKean-Vlasov SDEs are related to the mean field interacting diffusions, namely systems of the form
\begin{align*}
\rmd X^i_t = b(X^i_t,\frac{1}{N}\sum_{i=1}^N\delta_{X^i_t})\rmd t +\rmd W^i_t,\quad i=1,\ldots N,
\end{align*}
where $W^i$ are independent Brownian motions. By classical resuls, e.g. \cite{sznitman1991topics, Mel1996, MR780770}, if $b$ is bounded and smooth (smoothness with respect to the measure variable is understood in the sense of Wasserstein distance), then the McKean-Vlasov SDE \eqref{eq:McKVla_no_bdry} is pathwise well-posed and, as $N\to\infty$, the empirical measure $\frac{1}{N}\sum_{i=1}^N\delta_{X^i}$ converges to the law of the solution $\bar X$ to the McKean-Vlasov SDE. This convergence result is a law of large numbers type result and is related to the asymptotic independence of the particles, the so-called propagation of chaos, see e.g. \cite{Szn1984}. The Fokker-Planck equation associated with the McKean-Vlasov SDE \eqref{eq:McKVla_no_bdry}, namely the equation for $\text{Law}(\bar X_t)$, is nonlinear, see Section \ref{sec: McKean-Vlasov pde}.

SDEs on a domain $\bar D\subseteq \mr^m$ with reflecting boundaries take the form
\begin{align}\label{eq:SDE_refl}
&\rmd X_t = b(X_t)\rmd t +\rmd W_t -\rmd k_t,\\
&X_t\in \bar{D},\quad \rmd k_t=n(X_t)\rmd |k|_t,\quad \rmd |k|_t = 1_{X_t\in \partial D}\rmd |k|_t,
\end{align}
where $W$ is a Brownian motion (we do not consider general diffusion coefficients) and $n(x)$ is the outer normal to $D$ at $x$; $|k|$ represents the total variation process associated with $k$. The solution is a couple $(X,k)$ and $-\rmd k$ represents a ``kick'', in the inward normal direction $-n(X)$, that the diffusion $X$ receives anytime it reaches the boundary, and that makes $X$ stay in the domain $\bar D$. Pathwise well-posedness for the SDE \eqref{eq:SDE_refl} has been proved under quite general conditions, see e.g. \cite{LioSzn1984,Tan1979}. The Fokker-Planck equation associated with the SDE \eqref{eq:SDE_refl} has Neumann-type boundary conditions, see Section \ref{sec:RDb}.

To our knowledge, the first work to deal with both McKean-Vlasov SDEs and reflecting boundaries is \cite{Szn1984}: there pathwise well-posedness is proved for the SDE
\begin{align*}
&\rmd \bar X_t = b(\bar X_t,\text{Law}(\bar X_t))\rmd t +\rmd W_t -\rmd \bar k_t,\\
&\bar X_t\in \bar{D},\quad \rmd \bar k_t=n(\bar X_t)\rmd |\bar k|_t,\quad \rmd |\bar k|_t = 1_{\bar X_t\in \partial D}\rmd |\bar k|_t,
\end{align*}
and convergence is shown for the particle system
\begin{align*}
&\rmd X^i_t = b(X^i_t,\frac{1}{N}\sum_{i=1}^N\delta_{X^i_t})\rmd t +\rmd W^i_t -dk^i_t, \quad i=1,\ldots N,\\
&X^i_t\in \bar{D},\quad \rmd k^i_t=n(X^i_t)\rmd |k^i|_t,\quad \rmd |k^i|_t = 1_{X^i_t\in \partial D}\rmd |k^i|_t,\quad i=1,\ldots N.
\end{align*}
Other works have studied McKean-Vlasov SDEs with reflecting boundaries, in more general contexts, especially in the context of backward SDEs, see e.g. \cite{Li2014}. However, in \cite{Szn1984} and in many of these works, the reflection is local, that is: at the level of the McKean-Vlasov SDE, $\text{Law}(\rmd \bar k_t)$ does not appear in the SDE; at the level of the particle system, $\rmd k^i$ acts only on the $i$-th particle $X^i$.

Closer to our work are the mean reflected (possibly backward) SDEs and related particle systems, introduced in \cite{BriEliHu2018,BriChaGuiLab2020}, and their generalization, namely the SDEs with a constraint on the law and related particle systems, introduced in \cite{BriCarChaHu2020}. Roughly speaking, in these SDEs a reflecting boundary is imposed on the law of the process. The typical example of this type of SDEs is the following:
\begin{align}
\begin{aligned}\label{eq:SDE_constraint_law}
&\rmd \bar X_t = b(\bar X_t,\text{Law}(\bar X_t))\rmd t +\rmd W_t -\rmd \bar K_t\\
&\text{Law}(\bar X_t) \in \bar{\mathcal{D}},\quad \rmd \bar K_t \text{ deterministic, ``non-zero only when $\bar X_t$ is on the boundary of $\bar{\mathcal{D}}$''}.
\end{aligned}
\end{align}
For such systems, under suitable conditions, \cite{BriCarChaHu2020} proves well-posedness and particle approximation. As a particular case, taking $\bar{\mathcal{D}}=\{\rho\mid \int x\rho(\rmd x) \ge q\}$ for a given $q\in\mr$, the constraint becomes $\E[\bar X_t]\ge q$, which is morally comparable to our constraint $\E[\bar X_t]=q(t)$ in the last line of \eqref{eq:McKVla_intro}. Due to the assumptions on $\bar{\mathcal{D}}$ (which must have a non-empty ``interior''), the condition $\E[\bar X_t]=q(t)$ is not covered by \cite{BriCarChaHu2020}, but this is not a big limitation: the condition $\E[\bar X_t]\ge q$ is actually more difficult to take into account than $\E[\bar X_t]=q(t)$, which gives an explicit form for $\bar K$ and makes the SDE a classical McKean-Vlasov SDE. However, compared to our equation \eqref{eq:McKVla_intro}, the restriction to deterministic $\bar K_t$ in \cite{BriCarChaHu2020} does not allow to consider reflecting boundaries for the process $\bar X$ (for reflecting boundaries on $\bar X$, the reflection $\rmd \bar k$ is not deterministic). When one adds reflecting boundaries in \cite{BriCarChaHu2020}, additional difficulties come into play, see Remark \ref{rmk:penalization}.

Probably the closest work to ours is \cite{Jab2017}. This paper considers a more general case than \cite{BriCarChaHu2020}, in particular removing from \eqref{eq:SDE_constraint_law} the requirement that $\bar K$ is deterministic. In particular, taking
\begin{align}
\bar{\mathcal{D}}=\{\rho\mid \int x\rho(\rmd x) \ge q,\,\text{supp}(\rho)\subseteq [0,1]\}\label{eq:constraint_similar}
\end{align}
the constraint on $\bar X$ in \eqref{eq:SDE_constraint_law} becomes
\begin{align*}
\E \bar X_t \ge q,\quad \bar X_t \in [0,1],
\end{align*}
which is morally similar to our constraints $\E \bar X_t = q(t)$, $\bar X_t \in [0,1]$ in \eqref{eq:McKVla_intro}. The work \cite{Jab2017} constructs a weak solution to the SDE \eqref{eq:SDE_constraint_law} (without the requirement of deterministic $\bar K$, in particular including condition \eqref{eq:constraint_similar}) by a penalization approach. However it does not show uniqueness, nor it considers the related particle system.

The work \cite{Bar2020} studies a system of $N$ Brownian particles $X^i$ hitting a Newtonian moving barrier $Y$. For this system the paper proves the convergence, as $N\to \infty$, to a McKean-Vlasov type SDE, whose associated Fokker-Planck equation solves a free-boundary problem. Now, in the frame of the moving barrier, that is taking $Z^i_t=X^i_t-Y_t$, the system in \cite{Bar2020} is similar to our model \eqref{eq:particle_eq_intro}, without the drift $\mu$, but with one important difference: in the expression \eqref{eq:K_k} for $\rmd K^N$, the term $\frac{1}{N}\sum_{i=1}^N \rmd k^i$ is replaced in \cite{Bar2020} by $\frac{1}{N}\sum_{i=1}^N k^i \rmd t$ (times some constant). In particular, unlike here, the term $\rmd K^N$ becomes of bounded variation in time in \cite{Bar2020}.

The paper \cite{DjeEliHam2019} studies the case of backward SDEs with reflecting boundaries depending both on the diffusion process $\bar X$ and on the law of $\bar X$, showing well-posedness for this type of SDEs and convergence for the corresponding penalization scheme. However, by the precise assumptions in \cite{DjeEliHam2019}, a condition of the form $\E[\bar X_t]\ge q$ or $=q$ cannot be taken in \cite{DjeEliHam2019}.

Finally, we mention the works \cite{CacCarPer2011,DelIngRubTan2015} and \cite{HamLedSoj2019}: they deal with systems of interacting diffusions, which arise respectively in neuroscience and in finance, and include also a nonlocal effect of boundaries, though the boundaries are not reflecting. More precisely, when one or more particles hit the boundary, the other particles make a jump proportional to the number of particles hitting the boundary.

\textbf{Novelty of our work}. The main feature of our model is the combination of reflecting boundaries and nonlocal interaction. At the level of the McKean-Vlasov SDE \eqref{eq:McKVla_intro}, this combination appears in the formula \eqref{eq:K_k_bar} for the interacting term $\rmd \bar K$, which contains the term $\rmd \bar k$. At the level of the particle system \eqref{eq:particle_eq_intro}, this fact corresponds to an oblique reflection for $X^{(N)}=(X^{1,N},\ldots X^{N,N})$, where the direction of reflection depends on the empirical measure $\frac{1}{N}\sum_{i=1}^N \delta_{X^{i,N}_t}$. As explained before, to our knowledge, this kind of systems is studied only in \cite{Jab2017} (which shows an existence result for the McKean-Vlasov SDE).

More specifically, in our model \eqref{eq:particle_eq_intro}, the nonlocal interaction comes from the condition $\frac{1}{N}\sum_{i=1}^N X^{i,N}_t =q(t)$ and keeps the direction of reflection for $X^{(N)}$ on the iperplane $\{x\mid \frac{1}{N}\sum_{i=1}^N x^i=0\}$. Intuitively, when a particle $X^i$ hits the boundary and receive a ``kick'' $-\rmd k^i$, then the other particles receive a kick $\frac{1}{N}\rmd k^j$ in the opposite direction, so that the average of the particles remains $q(t)$.

As we will explain below, our proof relies strongly on the constraint $\frac{1}{N}\sum_{i=1}^N X^{i,N}_t =q(t)$ and so on the specific form of the interaction. The question of well-posedness and particle approximation for a more general dependence of $\rmd \bar K$ on $\rmd \bar k$, or equivalently, of the direction of reflection of $X^{(N)}$ on the empirical measure, remains open.

\textbf{Method of proof}. The main idea of the proof is that both $\rmd \bar k$ and $\rmd \bar K$ in \eqref{eq:McKVla_intro} act as projectors. Precisely, $\rmd \bar k$ acts as projector on the set of paths staying in $[0,1]$: indeed, if $(\bar X,\bar k^{\bar X},\bar K^{\bar X})$ and $(\bar X,\bar k^{\bar X},\bar K^{\bar X})$ are two solutions to \eqref{eq:McKVla_intro}, then $(\bar X-\bar Y)\cdot \rmd \bar k^{\bar X}\le 0$. The term $\rmd \bar K$ acts as projector in $L^2(\Omega)$ (where $(\Omega,\mathcal{A},P)$ is the underlying probability space) on the space of processes $Z$ with average $\E[Z_t]=q(t)$: indeed, if $(\bar X,\bar k^{\bar X},\bar K^{\bar X})$, $(\bar X,\bar k^{\bar X},\bar K^{\bar X})$ are two solutions, then $\E[(\bar X-\bar Y)\rmd \bar K^{\bar X}]=0$. This idea of projectors allows to show uniqueness for the McKean-Vlasov SDE \eqref{eq:McKVla_intro} easily.
This idea is also behind the proof of convergence of the particle system \eqref{eq:particle_eq_intro}.

Concerning the convergence of the particle system \eqref{eq:particle_eq_intro}, we use a pathwise approach introduced by Tanaka in \cite{MR780770} and revisited in \cite{CDFM2020} (and further developed in \cite{MR3299600,BaiCatDel2020} in the rough path context): for any fixed $\omega\in \Omega$, the particle system \eqref{eq:particle_eq_intro} can also be viewed as a McKean-Vlasov equation \eqref{eq:McKVla_intro}, where however the law on the driving signal $\bar W$ is not the Wiener measure, but the random empirical measure $L^{W,N}(\omega) = \frac{1}{N}\sum_{i=1}^N \delta_{W^i(\omega)}$. Under this viewpoint, the term $\rmd K^N$ is also a projector in $L^2$ on the space of processes with average $q(t)$, but where the underlying measure is the empirical measure $L^{W,N}(\omega)$ instead of the Wiener measure. Intuitively then, since $L^{W,N}(\omega)$ converges $P$-a.s. to the Wiener measure, we expect that $\rmd K^N$ should converge to the projector under the Wiener measure, that is $\rmd \bar K$; this should imply the convergence of the particle system to the McKean-Vlasov SDE \eqref{eq:McKVla_intro}.

However, a direct proof based only on this pathwise approach seems not easy. Indeed, to make this argument work, one needs to create an optimal coupling (in the sense of Wasserstein distance) between the Wiener measure and the random empirical measure $L^{W,N}(\omega)$, and such coupling does not have a Gaussian structure. Having a Gaussian measure on the driving signal allows to use classical stochastic analysis tools like It\^o formula; such tools give in turn uniform $BV$ estimates on $\rmd K^N$, which are also needed in the proof. Moreover, the pathwise argument gives a convergence only of the one-time marginals, that is, convergence of $\frac{1}{N}\sum_{i=1}^N\delta_{X^{i,N}_t(\omega)}$ (as a random measure on $[0,1]$) to the law of $\bar{X}_t$ for every fixed $\bar X_t$, and not convergence of $\frac{1}{N}\sum_{i=1}^N\delta_{X^{i,N}_{\cdot}(\omega)}$ as a measure on the path space $C([0,T];[0,1])$.
	
For these reasons, we use first a tightness argument. Namely we show uniform (in $N$) $BV$ and H\"older type bounds on the solution to the particle system \eqref{eq:particle_eq_intro}, which give a tightness result; we show then that any limit point of \eqref{eq:particle_eq_intro} satisfies the McKean-Vlasov SDE \eqref{eq:McKVla_intro}, obtaining at once convergence of the particle system and existence for the McKean-Vlasov SDE itself. Once we have these uniform $BV$ bounds and the existence for the McKean-Vlasov SDE, we can then use the pathwise argument explained before. From this pathwise argument, we also get the rate of convergence $O(1/\sqrt{\log(N)})$ for $\frac{1}{N}\sum_{i=1}^N\delta_{X^{i,N}_{t}(\omega)}$.

The fact that $\rmd \bar k$ acts like a projector is classical and used since at least \cite{LioSzn1984}. However, combining this fact with standard fixed-point arguments for McKean-Vlasov SDEs, as in \cite{sznitman1991topics,Szn1984}, seems not easy in presence of nonlocal effects of boundaries as here. The reason  in our model is that $\E[\rmd \bar k]$ is just a $BV$ and continuous term in time (not Lipschitz-continuous) without a distinguished sign. Even the use of the Lipschitz bounds from \cite{MR1110990}, on the $\sup$ norm of the reflecting term $\bar k$ in terms of the driving signal, seems not too helpful. For our model, the fact that $\rmd \bar K$ acts also as projector allows to overcome this difficulty.


We remark that the action of $\rmd \bar K$ as projector on the processes with prescribed average, as well as some tricks used in the proof of uniform H\"older bounds for the particle system, are quite specific to our model. For extension to more general nonlocal effects of the reflecting boundary, other methods may be useful, which we do not explore here, for example the penalization method (e.g. \cite{Men1983}, used in \cite{Jab2017}), the approach based on Lions' derivative (e.g. \cite{Lio2008,CarDelBookI}, used in \cite{BriCarChaHu2020}), pathwise approaches (e.g. \cite{DeyGubHofTin2019,Aid2016,FerRov2013}), PDE-based or singular interaction methods (see the paragraph below). 


\textbf{The PDE and singular interaction viewpoint}. The Fokker-Planck equation \eqref{eq:FPE_intro} associated with the McKean-Vlasov SDE \eqref{eq:McKVla_intro}, that is the equation for the probability density function (pdf) $u(t,\cdot)$ of $\bar X_t$, is a nonlinear nonlocal PDE. From this PDE \eqref{eq:FPE_intro}, we can get another expression for $\rmd \bar K$ in terms of $u$:
\begin{align*}
\dot{\bar K}_t = \dot{q}(t) +\int_0^1 \mu(x)u(t,x)\rmd x +\frac{\sigma^2}{2}(u(1)-u(0)).
\end{align*}
The nonlocal effects of the reflecting boundary appear as an additional drift term (here $u(1)-u(0)$) depending on the values of the probability density function $u$ at the boundary. Hence our model can be interpreted as a McKean-Vlasov SDE with reflecting boundaries and singular interaction at the boundary, and one may try to use PDE methods or an approach to singular interaction to deal with our model.

The literature about McKean-Vlasov SDEs with singular interaction and about PDE-based approach to McKean-Vlasov SDEs is large, though we are not aware of a work which can cover easily our model. We only mention some references related to our model. We mention \cite[Chapter II]{sznitman1991topics}, which deals with viscous Burgers equation as a McKean-Vlasov SDE with Dirac delta interaction (without boundaries), in particular the McKean-Vlasov SDE is also driven by the probability density function of the solution. We also mention \cite{BosJab2011,BosJab2015,BosJab2018} which study systems of interacting second-order diffusions (that is, SDEs for the acceleration of the particles) in a domain. Such systems contain a form of singular interaction and a form of reflection at the boundaries, though both interaction and reflection are of different type than in our model. Finally we mention \cite{Kol2007} for an approach based on semigroup theory to McKean-Vlasov SDEs and particle approximation.

In this paper we do not explore the PDE viewpoint and we give in Section \ref{sec:RDb} a formal argument, without any rigorous proof, to show that \eqref{eq:FPE_intro} is indeed the Fokker-Planck equation for the SDE \eqref{eq:McKVla_intro}.



\textbf{Organization of the paper}. The paper is organized as follows: In Section \ref{sec:PDEreview} we show the formal link, without rigorous proofs, between SDEs and Fokker-Planck equations, in presence of boundaries and mean-field interaction. In Section \ref{sec:setting} we give the precise setting and the main results. The proofs of these results are given in Section \ref{sec:proof}. Finally, in the Appendix \ref{app:B} we show well-posedness for the particle systems \eqref{eq:particle_eq_intro}.

\subsection{Acknowledgements}

The authors acknowledge support from Einstein Center for Mathematics Berlin through MATHEON projects C-SE8 and C-SE17 `Stochastic methods for the analysis of lithium-ion batteries'. M.C. and M.M. acknowledge support from Hausdorff Research Institute for Mathematics in Bonn under the Junior Trimester Program `Randomness, PDEs and Nonlinear Fluctuations'. This project has received funding from the European Research Council (ERC) under the European Union’s Horizon 2020 research and innovation programme (grant agreement No. 683164, PI P.K.Friz.) 

We thank Boualem Djehiche, Jean-Francois Jabir, John Schoenmakers and Andreas Sojmark for pointing out relevant references on McKean-Vlasov SDEs on bounded domains and suggesting possible alternative approaches to our problem.

\section{Review on PDEs and diffusion processes}  \label{sec:PDEreview} 

In this section we revisit the link between second-order PDEs and associated SDEs, both in presence of boundary and with nonlinearity, as the PDE \eqref{eq:FPE_intro} we consider here; we took inspiration from \cite{Son2007}. Our aim here is not to give rigorous results but to provide an easy, yet clear ``translator'' between SDEs and PDEs, which applies, but is not restricted, to our case and shows in particular why \eqref{eq:McKVla_intro} is the SDE corresponding to \eqref{eq:FPE_intro}. For this reason, we keep all the computations at a formal level, without any rigorous proof.

In the following we focus our attention on the 1D case, mostly for simplicity. We take $b:I\rightarrow\mr$ a given vector field on $\mr$ or on an interval $I$ of $\mr$ when specified, and $\sigma>0$ positive constant; $W$ is a real Brownian motion. For two functions $f,g:I\rightarrow \mr$, we call
\begin{align*}
\lan f,g \ran = \int_I f(x)g(x)\rmd x
\end{align*}
their $L^2$ scalar product.

\subsection{Diffusion, forward and backward PDEs} \label{sec:PDE1o1}

It is well-known that the (forward) PDE 
\begin{align*}
\partial_t p = (\sigma^2/2) \partial_y^2 p - \partial_y(bp)   \text{ for $t>s$,} 
\end{align*}
with time-$s$ initial data $\delta_x$, models the evolution of the transition density function $p=p(s,x;t,y)$ of the diffusion process $X$, given by
\begin{align*}
      \rmd X_t = \sigma \rmd W_t + b (t,X_t)\rmd t, \ X_s = x \ .
\end{align*}
More generally, if $u$ is the solution
to this forward PDE with time-$s$ initial data $u_s$, which is assumed to be a probability density function, then $u$ is the pdf of the solution of the same SDE but initial law $u_s(x) \rmd x$.

The dual viewpoint will be important. Consider the (backward) PDE
\begin{align*}
- \partial_s v = (\sigma^2/2)\partial_x^2 v + b \partial_x v
\end{align*}
with terminal data $v(t,.) = \Psi$. Assuming $v$ to be regular enough, 
It\^o's formula gives a representation of the backward PDE solution $v$ as follows:
\begin{align} \label{v1}
 v (s,x) = E [ \Psi (X_ t) | X_s = x] .
\end{align} 
By a simple formal computation, one shows that
\begin{align*}
     \frac{\rmd}{\rmd r}\lan p(s,x,r,\cdot),v(r,\cdot) \ran=0,
\end{align*}
which implies
\begin{align} \label{v2}
    v (s,x) = \int \Psi (y) p(s,x,t,y) \rmd y .
\end{align}
At last, comparing \eqref{v1} and \eqref{v2} shows that $p(s,x,t,y)\rmd y$ is indeed the law of $X_t$, started at $X_s=x$, as claimed in the beginning of this paragraph.

\subsection{Reflected diffusion and PDEs with boundary} \label{sec:RDb}

We now discuss the case of a spatial domain, with focus on the simple case $I=[0,1]$. A reflected SDE on the domain $[0,1]$ is an SDE of the form
\begin{align*}
      & \rmd X^\circ_t = \sigma \rmd W_t + b(X^\circ_t) \rmd t - \rmd k_t  \ , \ X^\circ_s = x, \\
      &  X^\circ_t \in \bar I \ \  \forall t \ge s, \\
      &\rmd |k| = 1_{X^\circ_t\in \{0,1\}} \rmd|k|,\ \ \rmd k = n(X^\circ_t)\rmd|k|.
\end{align*}
Here, $n(0)=-1, n(1)=+1$ are the outer normals of our domain $[0,1]$. The solution is a couple $(X,k)$ satisfying the above condition (it is implicitly assumed that $k$ has $BV$ paths). The last condition means that $k$ acts only when $X$ is on the boundary, giving a small ``kick'' so that $X$ does not leave the domain $I$.

From the PDE viewpoint, the transition density function $p^\circ(s,x;t,y)$ associated with $X$ is a solution to the following forward forward equation
\begin{align*}
       & \partial_t p^\circ = (\sigma^2/2) \partial_y^2 p^\circ - \partial_y(bp^\circ) \text{ in $I^\circ$ for $t>s$,} \\ 
       &  (\sigma^2/2) \partial_y p^\circ - b p^\circ = 0   \text{ at $\partial I$ for $t>s$} \\
       & p^\circ(s,x,s,\cdot)  = \delta_x ,
\end{align*}
More generally, if $u$ is the solution to this forward PDE with time-$s$ initial data $u_s$, assumed to be a probability density function, then $u$ is the pdf of the solution of the same SDE but initial law $u_s(x)\rmd x$. We show this fact formally using the dual viewpoint.
 
First step: Let $v^\circ$ be a regular solution to the dual backward equation, i.e. time-$t$ terminal value probelm with Neumann (no-flux) boundary data,
\begin{align*}
       & - \partial_sv^\circ  = (\sigma^2/2) \partial_x^2v^\circ + b \partial_x v^\circ  \text{ in $I^\circ$, for $s<t$ ,} \\ 
       & \partial_x v^\circ (s,0)  = \partial_x v^\circ (s,1) = 0   \text{ at $\partial I$, for all $s<t$ ,} \\
       & v(t,.)  = \Psi \ 
\end{align*}
with some (regular) time-$t$ terminal data $\Psi$. Then necessarily $v^\circ$ has the representation 
\begin{equation} \label{FKrefl}
v^\circ (s,x) = E[ \Psi(X^\circ_t) | X^\circ_s = x] .
\end{equation}
Indeed, It\^o formula gives
\begin{align*}
\rmd [v(s,X_s)] &= \partial_s v(X) \rmd s + \partial_x v(X) \rmd X +\frac{\sigma^2}{2} \partial_x^2 v(X) \rmd s\\
&= [(\partial_s +b\partial_x +\frac{\sigma^2}{2} \partial_x^2) v](X) \rmd s +\sigma\partial_x v(X) \rmd W -\partial_x v(X) \rmd k
= \sigma\partial_x v(X) \rmd W,
\end{align*}
where we have used the equation for $v$ to kill the first addend in the second line and the boundary conditions on $v$ to kill the term with $\rmd k$. Taking expectation, we get that $E[v(s,X_s)]$ is constant in $s$, which implies (\ref{FKrefl}).

Second step: Again with simple formal computation, one shows, 
\begin{align*}
    \frac{\rmd}{\rmd r}\lan p^\circ(s,x,r,\cdot),v^\circ(r,\cdot) \ran=0 ,
\end{align*}
which implies
\begin{align} \label{v4}
     v^\circ(s,x) = \lan p_t,\Psi \ran .
\end{align}
From (\ref{FKrefl}) and (\ref{v4}) we conclude
\begin{align*}
\lan p^\circ (s,x,t,\cdot),\Psi \ran  = E[ \Psi(X^\circ_t) | X^\circ_s = x] \ .
\end{align*}
 Since this is true for every regular $\Psi$, then $p^\circ$ is the law of $X^\circ$ conditional to $X^\circ_s = x$.

If $u^\circ$ is a solution to the same forward equation as above, but with generic initial condition $u_s$, integrating $p$ in $u_s(x)$ shows that $u$ is the pdf of the law of $X^{\circ, u}$, the process satisfying the same SDE but with initial law $u_s(x)\rmd x$.


\subsection{McKean-Vlasov diffusion, nonlinear mean-field PDEs}
\label{sec: McKean-Vlasov pde}

Here we introduce the McKean-Vlasov setting. For a given drift $\bar b :\mr\times \mc{P}(\mr)\rightarrow \mr$, where $\mc{P}(\mr)$ is the space of probability measures on $\mr$, we consider
\begin{align*}
\rmd \bar X_t = \sigma \rmd W_t + \bar b (\bar X_t, \text{Law}(\bar X_t))  \rmd t, \  \text{ Law}(\bar X_s) \text{ given}.
\end{align*}
Existence and uniqueness for regular bounded $\bar{b}$ (regularity with respect to the measure variable is usually in terms of the Wasserstein distance) are proved via a fixed-point argument on the map $(m_t)_{t\ge s}\mapsto (\text{Law}(X^m_t))_{t\ge s}$, where $X^m$ is the solution to the (classical) SDE
\begin{align*}
\rmd X^m_t = \sigma \rmd W_t + \bar b [ X^m_t,m_t]  \rmd t, \  \text{ Law}(X^m_s) = m_s.
\end{align*}
See Sznitman \cite[Theorem 1.1]{sznitman1991topics} for the classical case of $\bar{b}$ linear in the measure argument or, for instance, \cite{CDFM2020} for a more general $\bar{b}$.

The corresponding forward Fokker-Planck equation now takes the form
\begin{align}
\partial_t \bar u = (\sigma^2/2) \partial_y^2 \bar u - \partial_y [\bar b (y, \bar u_t(\cdot)) \bar u  ]\label{eq:bar_u}
\end{align}
with time-$s$ initial data $\bar u_s = \text{Law}(\bar X_s)$ (for notational sake, we blur the
difference between $u(t,.)$ and $u(t,y)dy$). Note that this is a nonlinear nonlocal PDE.

We again outline why $\bar u (t,.)$ is indeed the law of $\bar X_t$. Fix the family $\bar u := \{ \bar u (t,.): t \ge s \}$ and consider the (linear!) backward PDE
\begin{align*}
- \partial_s v = (\sigma^2/2) \partial_x^2 v + \bar b [x, \bar u (s,.) ] \partial_x v
\end{align*}
with terminal data $v_t = \Psi$. As in Section \ref{sec:PDE1o1}, we see that the corresponding forward PDE (which by construction coincides with \eqref{eq:bar_u}) yields the law of a diffusion process $X^{\bar u}$, that is $\text{Law}(X^{\bar u}_t) = \bar u (t,\cdot)$. But then
$X^{\bar u}$ solves the McKean-Vlasov SDE, therefore (by uniqueness for the McKean-Vlasov SDE) $X^{\bar u} = \bar{X}$ and so $\bar u_t$ is the pdf of the marginal $\bar{X}_t$ at time $t$.

\subsection{Reflected McKean-Vlasov diffusion, mean-field PDE with boundary}

We can consider the simplest case of reflected McKean-Vlasov SDE on $I$, as in \cite{Szn1984}, namely
\begin{align*}
& \rmd \bar X_t = \sigma \rmd W_t + \bar b [\bar X_t, \text{Law}(\bar X_t)]  \rmd t - \rmd \bar k_t, \  \text{ Law}(\bar X_s) \text{ given},\\
& \bar X_t \in \bar I \ \  \forall t \ge s, \\
& \rmd|\bar k| = 1_{\bar X_t\in \{0,1\}} \rmd|\bar k|,\ \ \rmd \bar k = n(\bar X_t)\rmd|\bar k|.
\end{align*}

Adapting the arguments in the two previous sections, one shows that the probability density function of $\bar X_t$ is given via following non-linear nonlocal forward PDE
\begin{align*}
& \partial_t \bar u_t = (\sigma^2/2) \partial_y^2 \bar u - \partial_y[\bar b [ y, \bar u_t (.)] \bar u] \text{ in $I^\circ$ for $t>s$\ ,} \\ 
&  (\sigma^2/2) \partial_y \bar u - \bar b [ \cdot ,  \bar u_t (.)] \bar u = 0   \text{ at $\partial I$ for $t>s$} ,
\end{align*}
with time-$s$ initial data $\bar u_s = \text{Law}(\bar X_s)$.

\subsection{Interaction coming from the boundaries}

As explained in the introduction, the SDE \eqref{eq:McKVla_intro} does not fall in the previous class. Indeed, the drift depends not only on the law of $\bar X_t$ but also on the law of $\rmd \bar k_t$. We will not treat here the general case of drifts depending on the law of $\rmd \bar k$,
but we focus our attention on our model.

We claim that the forward Fokker-Planck equation associated with \eqref{eq:McKVla_intro} is \eqref{eq:FPE_intro}: we show formally that, if $u$ satisfies the PDE \eqref{eq:FPE_intro}, with initial condition $u_0$, then $u(t,x)$ is the probability density function of the random variable $\bar X_t$ solving \eqref{eq:McKVla_intro}, with initial law $u_0$. The formal proof puts together the arguments for boundary problems and McKean-Vlasov SDEs with the additional difficulty of interaction coming from the boundary, for which we will use the constraint on the average in \eqref{eq:FPE_intro}.

Take $\dot{\bar K}$ as in the PDE \eqref{eq:FPE_intro}, call $(X^{\bar K},k^{\bar K})$ the solution of the reflecting SDE
\begin{align*}
&\rmd X^{\bar K} = \sigma \rmd W_t + (-  \mu ( X^{\bar K}) + \dot{\bar K}(t) ) \rmd t - \rmd k^{\bar K}, \  \text{ Law}(X^{\bar K}) = u_0,\\
& X^{\bar K}_t \in \bar I \ \  \forall t \ge 0, \\
& \rmd|k^{\bar K}| = 1_{X^{\bar K}_t\in \{0,1\}} \rmd|k^{\bar K}|,\ \ \rmd k^{\bar K} = n(X^{\bar K}_t)\rmd|k^{\bar K}|.
\end{align*}
As a consequence of Section \ref{sec:RDb} (applied with given $\dot{\bar K}$), the law of $X^{\bar K}_t$ must be given by $u(t,\cdot)$, i.e. the PDE solution to \eqref{eq:FPE_intro}. It remains to show that $X^\Lambda = \bar X$, the solution to the McKean-Vlasov \eqref{eq:McKVla_intro}. To this end, note that $\E[X^{\bar K}_t] = \int xu(t,x) \rmd x  = q(t)$, using the basic constraint in \eqref{eq:FPE_intro}. This shows that $X^{\bar K}$ is a solution to \eqref{eq:McKVla_intro} and, by uniqueness of this equation (formally, see also Theorem \ref{thm:one}), we conclude that $X^{\bar K} = \bar X$.

\section{The setting and the main results}\label{sec:setting}

\subsection{The particle system}

In the following, we consider a probability space $(\Omega,\mc{A},P)$ and independent Brownian motions $W^i$, $i=1\ldots N$, on a filtration $(\mc{F}_t)_t$ (satisfying the standard assumption). We are given a function of space $\mu:[0,1]\rightarrow \mr$ and a function of time $q:[0,T]\rightarrow[0,1]$. The assumptions on $\mu$ and $q$ will be given later. The noise intensity $\sigma$ is assumed to be constant (possibly $0$). 

We consider the system of $N$ interacting particles:
\begin{align}
\begin{aligned}\label{SDE_single}
&\rmd X^{i,N}_t = (-\mu(X^i_t)+\frac{1}{N}\sum^N_{j=1}\mu(X^{j,N}_t))\rmd t +\dot{q}(t)\rmd t +\sigma\rmd W^i -\frac{1}{N}\sum^N_{j=1}\sigma\rmd W^j_t -\rmd k^{i,N}_t +\frac{1}{N}\sum^N_{j=1}\rmd k^{j,N}_t, \ \ i=1,\ldots N,\\
&X^{i,N}\in C([0,T];[0,1]),\ k^i\in C([0,T];\mr) \text{ a.s.},\ \ i=1,\ldots N,\\
&\rmd|k^{i,N}|_t = 1_{X^{i,N}_t\in \{0,1\}}\rmd|k^{i,N}|_t,\ \ \rmd k^{i,N}_t = n(X^{i,N}_t)\rmd|k^{i,N}|_t,\ \ i=1,\ldots N.
\end{aligned}
\end{align}

Here $n$ is the outer normal vector of the domain $]0,1[$ and $|k^{i,N}|$ is the total variation process associated with $k^{i,N}$ (mind that it is not the modulus of $k^{i,N}$).
A solution is a couple $(X^{(N)},k^{(N)})= (X^{i,N},k^{i,N})_{i=1,\ldots N}$ of a $(\mc{F}_t)_t$-progressively measurable continuous semimartingale $X^{(N)}$ and a $(\mc{F}_t)_t$-progressively measurable $BV$ process $k^{(N)}$, satysfying the above system. Without loss of generality, we can assume $k^{(N)}_0=0$. We will often omit the second superscript $N$ (which denotes the number of particles) when not needed.

This system is the exact formulation of the interacting particle system \eqref{eq:particle_eq_intro}, with the term $\rmd K^N$ in \eqref{eq:particle_eq_intro} given by
\begin{align*}
\rmd K^N = \dot{q}(t)\rmd t +\frac{1}{N}\sum^N_{j=1}\mu(X^j_t) \rmd t -\frac{1}{N}\sigma\sum^N_{j=1}\rmd W^j +\frac{1}{N}\sum^N_{j=1}\rmd k^j.
\end{align*}
The term $-\rmd k^i$ represents the reflection at the boundary of $[0,1]$, while the term $\rmd K^N$ is independent of $i$ and ensures that the empirical average $N^{-1}\sum^N_{i=1}X^i_t$ stays equal to $q(t)$.

The system \eqref{SDE_single} can be interpreted as an SDE for $X$ on $[0,1]^N$ with oblique reflecting boundary conditions, where the direction of reflection keeps $X^{(N)}$ in the moving hyperplane $H_t=\{x\in\mr^N \mid \frac{1}{N}\sum^N_{i=1}x^i = q(t)\}$. Another interpretation of this system is as an SDE on $H_t \cap [0,1]^N$ with normal boundary condition, where the domain, in the frame of $H_t$, is a moving convex polygon and the normal reflection is in the frame of $H_t$. This interpretation is used in the proof of well-posedness of the system \eqref{SDE_single} (Proposition \ref{lem:wellpos_particle}).

We work under the following assumptions on $\mu$ and $q$ (and $X_0^{(N)}$):

\begin{condition}\label{assumptions_mu}
\begin{enumerate}
\item[i)] The function $-\mu$ is $C^2$ on $]0,1[$ and one-side Lipschitz-continuous, namely: there exists $c\ge0$ such that, for every $x$, $y$ in $]0,1[$,
\begin{align}
-(\mu(x)-\mu(y))(x-y)\le c|x-y|^2, \quad \forall x,y\in ]0,1[.\label{eq:mu_oneside_Lip}
\end{align}
\item[ii)] The function $\mu$ satisfies
\begin{align*}
\sup_{x\in ]0,1/2[}|x||\mu(x)| +\sup_{x\in ]1/2,1[}|1-x||\mu(x)| <+\infty.
\end{align*}
\item[iii)] There exists $0<\rho<1/2$ such that
\begin{align}
\text{sign}(x-1/2)\mu(x)\ge 0,\quad  \forall x\in ]0,\rho[\cup ]1-\rho,\rho[;\label{eq:mu_bd}
\end{align}
moreover $\mu(0)=\mu(1)=0$.
\end{enumerate}
\end{condition}

\begin{condition}\label{assumptions_q}
The map $q$ is a Lipschitz-continuous function of time (in particular $\dot{q}(t)$ exists for a.e. $t$) and there exists $0<\xi<1$ such that $\xi\le q(t)\le 1-\xi$ for every $t$.
\end{condition}

\begin{condition}\label{assumptions_x0N}
The (possibly random) initial datum $X_0^{(N)}$ is $\mc{F}_0$-measurable and verifies $0\le X_0^{i,N} \le 1$ for $i=1,\ldots N$ and $\frac{1}{N}\sum^N_{i=1} X^{i,N}_0 = q(0)$ $P$-a.s..
\end{condition}

The typical example we have in mind for $\mu$ is the derivative of a double-well potential, with logarithmic divergence at the boundary, and the typical example for $q$ is a piecewise linear continuous function which does not touch $0$ nor $1$. These examples are used in the battery model from \cite{GGMFD2018}.

While Condition \ref{assumptions_mu}-(i) on $\mu$ and Condition \ref{assumptions_q} on $q$ are structural assumptions of our model, Condition \ref{assumptions_mu}-(ii) seems not really necessary: if for example $\mu$ diverges like $1/x^\alpha$ for some $\alpha>1$, we would expect that the system does not even touch the boundary, hence classical McKean-Vlasov approach should apply, but for technical reasons our proof does not apply to this situation, see Remark \ref{rmk:tech_1}. Condition \ref{assumptions_mu}-(iii) is also technical and we expect that it can be removed without too much effort, see Remark \ref{rmk:tech_2}.

Actually we do not work directly with the system \eqref{SDE_single} but, to avoid possible singularity of $\mu$ at the boundary, we take a regularization $\mu^\eps$, $C^2$ on the closed domain $[0,1]$, with $\mu^\eps=\mu$ on $[\eps,1-\eps]$ and $|\mu^\eps|\le |\mu|$ on $]0,1[$ and verifying both the one-side Lipschitz condition \eqref{eq:mu_oneside_Lip} and the condition \eqref{eq:mu_bd} uniformly in $\eps$. We then consider the system:
\begin{align}
\begin{aligned}\label{SDE_single_reg}
&\rmd X^i_t = (-\mu^\eps(X^i_t)+\frac{1}{N}\sum^N_{j=1}\mu^\eps(X^j_t))\rmd t +\dot{q}(t)\rmd t +\sigma\rmd W^i_t -\frac{1}{N}\sum^N_{j=1}\sigma\rmd W^j_t -\rmd k^i_t +\frac{1}{N}\sum^N_{j=1} \rmd k^i_t , \ \ i=1,\ldots N,\\
&X^i\in C([0,T];[0,1]),\ k^i\in C([0,T];\mr) \text{ a.s.},\ \ i=1,\ldots N,\\
&\rmd|k^i|_t = 1_{X^i_t\in \{0,1\}}\rmd|k^i|_t,\ \ \rmd k^i_t = n(X^i_t)\rmd|k^i|_t,\ \ i=1,\ldots N.
\end{aligned}
\end{align}
When we want to stress the dependence on $N$ and $\eps$, we write $X^{i,N,\eps}$ and $X^{(N,\eps)}=(X^1,\ldots X^N)$ and similarly for $k^{i,N,\eps}$, $k^{(N,\eps)}$.

\begin{remark}
Here and in the following, when we talk about pathwise uniqueness, resp.\ uniqueness in law, we refer to pathwise uniqueness of $X^{(N,\eps)}$, resp. of the law of $X^{(N,\eps)}$. Uniqueness of $X^{(N,\eps)}$ implies in turn uniqueness of $k^{(N,\eps)}-\E k^{(N,\eps)}$, but we do not make any uniqueness statement on $k^{(N,\eps)}$ itself.
\end{remark}

\begin{proposition}\label{lem:wellpos_particle}
Assume Conditions \ref{assumptions_q} and \ref{assumptions_x0N} and assume that $\mu^\eps$ is Lipschitz-continuous on $[0,1]$. Then there exists a solution to the particle system \eqref{SDE_single_reg} and this solution is pathwise unique in $X^{(N,\eps)}$.
\end{proposition}

The basic idea of the proof is simple: namely the SDE above is an SDE on the moving domain $H_t \cap [0,1]^N$ with normal boundary conditions. However the proof is slightly technical and postponed to Appendix \ref{app:B}.


\begin{remark}
Similarly to \eqref{SDE_single}, any solution to \eqref{SDE_single_reg} satisfies $\frac{1}{N}\sum^N_{i=1} X^i_t = q(t)$ for every $t$.
\end{remark}


\subsection{The McKean-Vlasov SDE}

In the following, we consider again a probability space $(\Omega,\mc{A},P)$ and a Brownian motion $W$ on a filtration $(\mc{F}_t)_t$ (satisfying the standard assumption); $\E$ denotes the expectation with respect to $P$. The functions $\mu$, $q$ and $\sigma$ are as in the previous subsection.

We consider the McKean-Vlasov SDE
\begin{align}
\begin{aligned}\label{meanfield_SDE_single}
&\rmd \bar X_t = -\mu(\bar X_t)\rmd t +\sigma\rmd W_t +\rmd \bar K_t -\rmd \bar k_t,\\
&\int^T_0\E[|\mu(\bar X_r)|]\rmd r<\infty,\ \ \E\int^T_0\rmd |\bar k|_r<\infty,\ \ \rmd \bar K_t = (\E[\mu(\bar X_t)] +\dot{q}(t))\rmd t +\E[\rmd \bar k_t],\\
&\bar X\in C([0,T];[0,1]),\ \bar k\in C([0,T];\mr) \text{ a.s.},\\
&\rmd|\bar k|_t = 1_{\bar X_t\in \{0,1\}}\rmd|\bar k|_t,\ \ \rmd \bar k_t = n(\bar X_t)\rmd|\bar k|_t.
\end{aligned}
\end{align}
Here again $n$ is the outer normal vector of the domain $]0,1[$ and $|\bar k|$ is the total variation process associated with $\bar k$ (not the modulus of $\bar k$). A solution is a couple $(\bar X,\bar k)$ of a $(\mc{F}_t)_t$-progressively measurable continuous semimartingale $\bar X$ and a $(\mc{F}_t)_t$-progressively measurable $BV$ process $\bar k$, satisfying the above equation. Without loss of generality, we can assume $\bar k_0=0$. We sometimes say that $\bar X$ is a solution if there exists a process $\bar k$ such that $(\bar X,\bar k)$ is a solution.

The assumptions on $\mu$ and $q$ remain unchanged with respect to the particle system. In the assumption on $\bar X_0$, here the empirical average is replaced by the average with respect to $P$.

\begin{condition}\label{assumptions_x0}
The (possibly random) initial datum $\bar X_0$ is $\mc{F}_0$-measurable and verifies $0\le \bar X_0 \le 1$ and $\E \bar X_0 = q(0)$.
\end{condition}

\begin{remark}\label{rmk:averageX}
As for the particle system, it is easy to see that (under Condition \ref{assumptions_x0}) any solution to \eqref{meanfield_SDE_single} satisfies $\E \bar X_t = q(t)$ for every $t$.
\end{remark}

In view of the proof of convergence of the particle system, it is convenient to write \eqref{meanfield_SDE_single} in the following equivalent way:
\begin{align}
\begin{aligned}\label{eq:meanfield_SDE_equiv}
&\rmd \bar X_t = \dot{q}(t)\rmd t +\sigma\rmd W_t +\rmd \bar Z_t -\E \rmd \bar Z_t,\\
&\int^T_0\E[|\mu(\bar X_r)|]\rmd r<\infty,\ \ \E\int^T_0\rmd |\bar k|_r<\infty,\ \ \rmd \bar Z_t = -\mu(\bar X_t)\rmd t -\rmd \bar k_t,\\
&\bar X\in C([0,T];[0,1]),\ \bar k\in C([0,T];\mr) \text{ a.s.},\\
&\rmd|\bar k|_t = 1_{\bar X_t\in \{0,1\}}\rmd|\bar k|_t,\ \ \rmd \bar k_t = n(\bar X_t)\rmd|\bar k|_t.
\end{aligned}
\end{align}

\subsection{The main results}

Our main results are well-posedness of the McKean-Vlasov SDE \eqref{meanfield_SDE_single} and convergence of the particle system \eqref{SDE_single_reg} to the McKean-Vlasov SDE as $N\rightarrow \infty$ and $\eps\rightarrow 0$.

\begin{theorem}\label{thm:one}
Take a probability space $(\Omega,\mc{A},P)$, a Brownian motion $W$ on a filtration $(\mc{F}_t)_t$ (satistying the standard assumption) and an initial condition $\bar X_0$, assume Conditions \ref{assumptions_mu}, \ref{assumptions_q}, \ref{assumptions_x0}. Then there exists a unique solution $(\bar X,\bar k)$ to the McKean-Vlasov SDE \eqref{meanfield_SDE_single}.
\end{theorem}

Let $\bar X$ be the solution to the McKean-Vlasov SDE \eqref{meanfield_SDE_single} with initial datum $\bar X_0$ and, for $\eps>0$, $N$ in $\mathbb{N}$, let $(X^{1,N,\eps},\ldots X^{N,N,\eps})$ be the solution to the particle system \eqref{SDE_single_reg} with initial datum $(X_0^1,\ldots X_0^N)$. For $E$ Polish space and $Y^i$ $E$-valued random variables, we consider the empirical measures $\frac{1}{N}\sum^N_{i=1} Y^i$ as $\mc{P}(E)$-valued random variable, where $\mc{P}(E)$ is the space of probability measures on $E$, endowed with the Borel $\sigma$-algebra with respect to the weak convergence (convergence against $C_b(E)$ test functions).

\begin{theorem}\label{thm:convergence}
Assume Conditions \ref{assumptions_mu}, \ref{assumptions_q}, Condition \ref{assumptions_x0} on $\bar X_0$ and Condition \ref{assumptions_x0N} on $(X_0^1,\ldots X_0^N)$. Assume also that $\frac{1}{N}\sum^N_{i=1}\delta_{X^{i,N}_0}$ converges in probability to $\text{Law}(\bar X_0)$ as $N\rightarrow\infty$. Then the sequence of empirical measures $\frac{1}{N}\sum^N_{i=1} \delta_{X^{i,N,\eps}}$ on $\mc{P}(C([0,T]))$ converges in probability to $\text{Law}(\bar X)$, as $\eps\rightarrow 0$ and $N\rightarrow\infty$.
\end{theorem}


Note that the convergence result of the particle system \eqref{SDE_single_reg} holds as $N\rightarrow\infty$, $\eps\rightarrow 0$ with no further restriction. In particular, one could send first $\eps$ to $0$ and then $N$ to $\infty$ to show the convergence of the original particle system \eqref{SDE_single}.

\begin{remark}
	\label{rmk: initial conditions}
The assumptions on the initial conditions may sound a bit rigid, in particular they cannot be satisfied taking $(X_0^1,\ldots X_0^N)$ i.i.d. copies of $\bar{X}_0$ (the empirical average is not $q(0)$ for a.e. $\omega$). However:
\begin{itemize}
\item An easy example of $(X_0^1,\ldots X_0^N)$ satisfying this constraint is given by taking $Y^i$ i.i.d. copies of a variable $\bar{X}_0$ with mean $q(0)$ and $X^i_0 = Y^i - \frac{1}{N}\sum^N_{j=1}Y^j +q(0)$: by the law of large number $\frac{1}{N}\sum^N_{j=1}Y^j$ tends to $q(0)=\E \bar{X}_0$ and so the empirical measure of $X^i_0$ tends to the law of $\bar{X}_0$  in probability.
\item The assumptions can easily be relaxed allowing $q(0)=q^N(0)$ to be random and dependent on $N$, but keeping deterministic increments $q^N(t)-q^N(0)$, with $q^N(0)$ tending to $q(0)$ in probability as $N\rightarrow \infty$. This allows to include the case of $(X_0^1,\ldots X_0^N)$ i.i.d. copies of $\bar{X}_0$.
\end{itemize}
\end{remark}

Finally, we give another convergence result and exhibit a rate of convergence for the time marginals.
We denote by $\mathcal{W}_{2,[0,1]}$ the $2$-Wasserstein distance on $[0,1]$.

\begin{proposition}
	\label{pro: convergence rate}
	Assume that $\mu$ is $C^2$ on $[0,1]$ so that we can take $\mu^{\eps} = \mu$ for every $\eps>0$. Assume the conditions of Theorem \ref{thm:convergence} and assume also that $X^{i,N}_0 = Y^i + \sum_{j=1}^N Y^j + q(0)$, where $(Y^i)_{i\in \mathbb{N}}$ is a sequence of independent and identically distributed random variables with law $\text{Law}(\bar{X}_0)$ (see Remark \ref{rmk: initial conditions}).
	Then we have the following rate of convergence:
	\begin{equation*}
	\mathbb{E} \left[\sup_{t\in[0,T]}\mathcal{W}_{2,[0,1]} ( \text{Law} ( \bar{X}_t), \frac1N \sum_{ i = 1 }^{ N } \delta_{X ^ { i , N} _t} ) \right]  
	= O(1/\sqrt{\log(N)})
	\qquad
	\mbox{as } N \to \infty.
	\end{equation*}
\end{proposition}

\section{The proof}\label{sec:proof}

\subsection{The strategy}

The strategy of the proof is as follows:
\begin{itemize}
\item We first prove uniqueness for the McKean-Vlasov SDE. For later use, we prove uniqueness among a larger class of solutions, namely possibly non-adapted processes. We also give a stability result with respect to the drift $\mu$.
\item For convergence of the particle system and existence of the McKean-Vlasov SDE, we prove uniform (in $N$ and $\eps$) $BV$ and H\"older estimates for $k^{i,N,\eps}$ and uniform H\"older estimates for $X^{i,N,\eps}$. These estimates in turn imply tightness for the empirical measures $\frac{1}{N}\sum^N_{i=1} \delta_{X^{i,N,\eps}}$ and more generally for $\frac{1}{N}\sum^N_{i=1}\delta_{(W^i,X^{i,N,\eps},-\int^\cdot_0\mu^\eps(X^{i,N,\eps}_r)\rmd r-k^{i,N,\eps})}$.
\item We then prove that any limit point of the empirical measures $\frac{1}{N}\sum^N_{i=1} \delta_{X^{i,N,\eps}}$ is the law of a (possibly non-adapted) solution to the McKean-Vlasov SDE. Uniqueness of the McKean-Vlasov SDE implies that the whole sequence of empirical measure converges to the law of the unique solution and that this solution is actually adapted.
\item Finally, we prove the rate of convergence using a pathwise approach. We first show that particle system \eqref{SDE_single_reg} can be interpreted as the McKean-Vlasov equation with a different measure on the inputs. The core of the proof is then a stability result of the McKean-Vlasov equation with respect to the inputs.
\end{itemize}

In the following subsections, we will use the letter $C$ to denote a positive constant, whose value may change from line to line; we will sometimes use $C_p$ to stress the dependence on $p$.

\begin{remark}\label{rmk:penalization}
	Here we comment about a possible alternative strategy, taken from \cite{BriEliHu2018,BriChaGuiLab2020}. One could try to apply the penalization method used in those works to equation \eqref{eq:McKVla_intro}, where in this case the penalized equation is a reflected equation.
	Reducing the problem to its bare bones and in order to make it as similar as possible to \cite{BriEliHu2018,BriChaGuiLab2020}, we can look at the following equation
	\begin{align}
		\begin{aligned}\label{eq:to_penalize_intro}
			&\rmd \bar X_t = \rmd W_t + \rmd \bar K_t -\rmd \bar k_t,\\
			&\bar X_t\le 1,\quad \rmd \bar k_t = n(\bar X_t)\rmd |\bar k|_t,\quad \rmd |\bar k|_t = 1_{\bar X_t=1}\rmd |\bar k|_t,\\
			&\E \bar X_t \geq \frac{1}{2},\quad \bar K \text{ deterministic}.
		\end{aligned}
	\end{align}
	The penalized version of equation \eqref{eq:to_penalize_intro} is the following
	\begin{align}
		\begin{aligned}\label{eq:penalized_intro}
			&\rmd \bar X^n_t = \rmd W_t + \rmd \bar K^n_t -\rmd \bar k^n_t,\\
			&\bar X_t\le 1,\quad \rmd \bar k_t = n(\bar X_t)\rmd |\bar k|_t,\quad \rmd |\bar k|_t = 1_{\bar X_t=1}\rmd |\bar k|_t,\\
			&\bar K^n_t := \int_{0}^{t}\varphi_n(\mathbb{E}[\bar X_s^n]-\frac{1}{2}) \rmd s
		\end{aligned}
	\end{align}
	where $\varphi_n(x) = r1_{x\leq -\frac{1}{n}} -nrx1_{-\frac{1}{n} < x \leq 0}$ and $r > 0$ is to be choose accordingly. Equation \eqref{eq:penalized_intro} is well posed, for every $n$ because of \cite{Szn1984}.
	
	The goal is now to construct a solution to equation \eqref{eq:to_penalize_intro} as a limit, for $n\to\infty$, of a sequence of solutions $\bar X^n$ to \eqref{eq:penalized_intro}. When proving that $\bar X^n$ is a Cauchy sequence in $L^2$ one gets
	\begin{equation*}
		\mathbb{E}[|\bar X^n_t - \bar X^m_t|^2] 
		\leq - 2 \int_{0}^{t} \mathbb{E}[\bar X_s^n] \varphi_m(\mathbb{E}[\bar X_s^m]-\frac{1}{2}) ds
		- 2\int_{0}^{t} \mathbb{E}[\bar X_s^m] \varphi_n(\mathbb{E}[\bar X_s^n]-\frac{1}{2}) ds.
	\end{equation*}
	Since $\varphi$ is bounded and non-negative, one could conclude by proving that $\mathbb{E}[\bar X_s^n] -\frac12 \geq  -\frac{c}{n}$ for some constant $c> 0$. By taking the expectation in equation \eqref{eq:penalized_intro} we get
	\begin{equation*}
		\mathbb{E}[\bar X^n_t] = \mathbb{E}[\bar X^n_s] + \int_{s}^{t}\varphi_n(\mathbb{E}[\bar X_u^n]-\frac{1}{2}) \rmd u - \mathbb{E}[\bar k^n_{t}-\bar k^n_{s}].
	\end{equation*}
	At this point we meet an additional difficulty with respect to \cite{BriCarChaHu2020}: in order to conclude as in the argument in \cite{BriCarChaHu2020}, one would need $\mathbb{E}[k_{\cdot}]$ to be Lipschitz function of time; however $k_{\cdot}$ is in general only of bounded variation.
	
	Maybe one could try to penalize both reflection terms. However, this is behind the scope of the present paper.
\end{remark}

\subsection{Uniqueness and stability}

In this Subsection we establish uniqueness and stability results for the McKean-Vlasov SDE \eqref{meanfield_SDE_single}. The following result proves the uniqueness part of Theorem \ref{thm:one}.

\begin{proposition}\label{prop:uniq_McKVla}
Assume Condition \ref{assumptions_mu}-(i) on $\mu$  and that $q$ is measurable bounded (Conditions \ref{assumptions_mu} and \ref{assumptions_q} in particular are enough). Assume also Condition \ref{assumptions_x0} on $\bar X_0$. Strong uniqueness holds for the McKean-Vlasov SDE \eqref{meanfield_SDE_single}. Moreover, if $\bar X$ and $\bar Y$ are two solution to \eqref{meanfield_SDE_single} starting from $\bar X_0$, $\bar Y_0$, with $\E[\bar X_0]=\E[\bar Y_0]=q(0)$, it holds for some $C>0$ (independent of $\bar X_0$ and $\bar Y_0$), for every $t$,
\begin{align*}
\E|\bar X_t-\bar Y_t|^2\le e^{2Ct}\E|\bar X_0-\bar Y_0|^2.
\end{align*}
\end{proposition}

\begin{proof}
It is enough to prove stability. We will use the superscripts $\bar X$, $\bar Y$ for the quantities $\bar K$, $\bar k$, ... associated with $\bar X$, $\bar Y$. By It\^o formula for continuous semimartingales \cite{revuz1999} we have
\begin{align*}
&\rmd|\bar X-\bar Y|^2\\
&= 2(\bar X-\bar Y)(-\mu(\bar X)+\mu(\bar Y))\rmd t +2(\bar X-\bar Y)\rmd \bar K^{\bar X} -2(\bar X-\bar Y)\rmd \bar K^{\bar Y}\\
&\ \ -2(\bar X-\bar Y)\rmd \bar k^{\bar X} +2(\bar X-\bar Y)\rmd \bar k^{\bar Y}.
\end{align*}
For the first addend, the one-side Lipschitz condition of $\mu$ implies
\begin{align*}
(\bar X-\bar Y)(-\mu(\bar X)+\mu(\bar Y))\le c|\bar X-\bar Y|^2.
\end{align*}
For the addends with $\bar k$, the orientation of $\bar k$ (as the outward normal) implies
\begin{align*}
-\int^t_0(\bar X-\bar Y)\rmd \bar k^{\bar X} \le 0
\end{align*}
and similarly for $(\bar X-\bar Y)\rmd \bar k^{\bar Y}$. For the addends with $K$, we take the expectation and use that $K$ is deterministic and that $E[\bar X_t]=E[\bar Y_t]=q(t)$ (see Remark \ref{rmk:averageX}): we obtain
\begin{align*}
\E\int^t_0(\bar X-\bar Y)\rmd K^{\bar X} = \int^t_0(\E[\bar X]-\E[\bar Y])\rmd K^{\bar X} =0.
\end{align*}
Putting all together, we get
\begin{align*}
\E|\bar X_t-\bar Y_t|^2 \le \E|\bar X_0-\bar Y_0|^2 + C\int^t_0\E|\bar X_r-\bar Y_r|^2\rmd r.
\end{align*}
We conclude by Gronwall inequality.
\end{proof}


\begin{proposition}
Assume Condition \ref{assumptions_mu}-(i) and that $q$ is measurable bounded. Let $\mu^n$ a sequence of functions, with uniformly bounded one-side Lipschitz constant, converging uniformly to $\mu$ on every compact subset of $]0,1[$, such that $|\mu^n|\le C|\mu|$ on $]0,1[$. Call $\bar X^n$, $\bar X$ the solutions to the SDE \eqref{meanfield_SDE_single} resp.\ with $\mu^n$, $\mu$ and with the same initial condition. Then it holds, as $n\rightarrow\infty$,
\begin{align*}
\sup_{t\in[0,T]}\E|\bar X_t-\bar X^n_t|^2 \rightarrow 0.
\end{align*}
\end{proposition}


\begin{proof}
By It\^o formula we have, proceeding as in the previous proof, we obtain
\begin{align*}
\rmd|\bar X-\bar X^n|^2 = 2(\bar X-\bar X^n)(-\mu(\bar X)+\mu^n(\bar X))\rmd t +2(\bar X-\bar X^n)(-\mu^n(\bar X)+\mu^n(\bar X^n))\rmd t +\rmd(\text{other terms}),
\end{align*}
where the other terms have non-positive expectation. For the second addend, the uniform one-side Lipschitz condition implies, for some $c>0$ independent of $n$,
\begin{align*}
(\bar X-\bar X^n)(-\mu^n(\bar X)+\mu^n(\bar X^n))\le c|\bar X-\bar X^n|^2.
\end{align*}
For the first addend, the integrability condition on $\mu$ in \eqref{meanfield_SDE_single} implies that, for every $\eps>0$, there exists $\delta>0$ such that
\begin{align*}
\E\int^T_01_{\bar X\notin [\delta,1-\delta]}|\mu(\bar X)|\rmd r<\eps
\end{align*}
and similarly for $\mu^n$ since $|\mu^n|\le C|\mu|$. By the uniform convergence of $\mu^n$ to $\mu$ on $[\delta,1-\delta]$, there exists $n_0$ such that, for every $n\ge n_0$, $|\mu^n-\mu|<\eps$ on $[\delta,1-\delta]$. Therefore we have
\begin{align*}
&\E\int^t_0(\bar X-\bar X^n)(-\mu(\bar X)+\mu^n(\bar X))\rmd r\\
&\le \E\int^t_01_{\bar X\in [\delta,1-\delta]}|\mu(\bar X)-\mu^n(\bar X)|\rmd r +\E\int^t_01_{\bar X\notin [\delta,1-\delta]}(|\mu(\bar X)|+|\mu^n(\bar X)|)\rmd r \le C\epsilon.
\end{align*}
Finally we obtain, for every $n\ge n_0$,
\begin{align*}
\E|\bar X_t-\bar X^n_t|^2 \le C\epsilon + C\int^t_0\E|\bar X_r-\bar X^n_r|^2\rmd r.
\end{align*}
We conclude again by Gronwall lemma.
\end{proof}

For the proof of convergence of the particle system, it is actually useful a slightly stronger uniqueness result, among a generalized class of solutions. Given a probability space $(\Omega,\mc{A},P)$ and a Brownian motion $W$ on it (with respect to its natural filtration), we call generalized solution a couple $(\bar X,\bar k)$ of $\mc{A}\otimes\mc{B}([0,T])$-measurable maps, satisfying the system \eqref{meanfield_SDE_single} (or equivalently \eqref{eq:meanfield_SDE_equiv}) $P$-a.s.,\ without any adaptedness condition; we also do not require $\mc{A}$ to be complete with respect to $P$. We also call weak generalized solution the object $(\Omega,\mc{A},W,\bar X,\bar k,P)$. Note that the system makes sense even without adaptedness, since the noise is additive. The difference with the usual concept of solution lies exactly in the lack of adaptability (and lack of completeness of the $\sigma$-algebra $\mc{A}$). We say that $\bar X$ is a generalized solution if there exists a $\mc{A}\otimes\mc{B}([0,T])$-measurable map $\bar k$ such that $(\bar X,\bar k)$ is a generalized solution.

\begin{lemma}
Assume Condition \ref{assumptions_mu}-(i) and that $q$ is measurable bounded. Assume also Condition \ref{assumptions_x0} on $\bar X_0$. Given $(\Omega,\mc{A},P)$ and $W$, uniqueness holds among generalized solutions.
\end{lemma}

\begin{proof}
Let $(\bar X,\bar k^{\bar X})$ and $(\bar Y,\bar k^{\bar Y})$ be two solutions. Then $\bar X-\bar Y$ is a $BV$ and continuous process satisfying $P$-a.e.
\begin{align*}
\rmd (\bar X-\bar Y) = (-\mu(\bar X)+\mu(\bar Y))\rmd t +\rmd (\bar K^{\bar X} - \bar K^{\bar Y}) +\rmd (\bar k^{\bar X} - \bar k^{\bar Y}).
\end{align*}
Each of the addends in the right-hand side is $BV$ and continuous, in particular we can fix $\omega$ (outside a $P$-null set in $\mc{A}$) and apply the chain rule to get the expression for the differential of $|\bar X-\bar Y|^2$. The rest of the proof goes as in the proof of Proposition \ref{prop:uniq_McKVla}.
\end{proof}

Another useful tool in view of particle convergence is Yamada-Watanabe principle, which, roughly speaking, states that strong uniqueness and weak existence imply uniqueness in law and strong existence. Since we are working in a slightly non-standard context, with McKean-Vlasov SDEs and with generalized solutions, we repeat the statements and the proofs for our case:

\begin{proposition}[Yamada-Watanabe, uniqueness in law]\label{YamWat_uniq}
Let $(\Omega^i,\mc{A}^i,P^i)$, $i=1,2$, be two probability spaces, with associated Brownian motions $W^i$ and generalized solutions $(\bar X^i,\bar k^i)$, $i=1,2$, such that $\text{Law}(\bar X^1_0)=\text{Law}(\bar X^2_0)$. Then the laws of $(W^1,\bar X^1)$ and $(W^2,\bar X^2)$ coincide.
\end{proposition}

\begin{proof}
We take $\hat{\Omega}=(C([0,T])\times \mr) \times C([0,T])^2 \times C([0,T])^2$, endowed with the Borel $\sigma$-algebra $\hat{\mc{A}}=\mc{B}(\hat{\Omega})$ (with respect to the uniform topology). We call $\hat{\omega}=((w,x_0),(\gamma^1,\kappa^1),(\gamma^2,\kappa^2))$ a generic element of $\Omega$. Let $P^{i,W^i,\bar X^i_0}$ be the conditional law of $(\bar X^i,\bar k^i)$ with respect to $W^i$ and $\bar X^i_0$, $i=1,2$. We take on $\mc{B}(\hat{\Omega})$ the probability measure $\hat P = P^{W,\bar X_0} \otimes P^{1,w,x_0} \otimes P^{2,w,x_0}$, where $P^{W,\bar X_0}$ is the product of the Wiener measure and the law of $\bar X^1_0$. We define $\hat W(\hat{\omega})=w$, $\hat X_0(\hat{\omega})=x_0$, $(\hat{X}^i(\hat{\omega}),\hat{k}^i(\hat{\omega}))=(\gamma^i,\kappa^i)$, $i=1,2$, the canonical projections. Now, for each $i=1,2$, the law of $(\hat W,\hat{X}^i,\hat{k}^i)$ is the law of $(W^i,\bar X^i,\bar k^i)$, in particular $\bar K^{\hat{X}^i}=\bar K^{\bar X^i}$. Therefore $(\hat{X}^i,\hat{k}^i)$, $i=1,2$, are two generalized solutions to \eqref{meanfield_SDE_single}, defined on the same probability space $(\hat{\Omega},\hat{\mc{A}},\hat P)$ with respect to the same Brownian motion $\hat W$ and with the same initial datum $\hat X^1_0=\hat X^2_0=\hat X_0$ $\hat P$-a.s.. By the uniqueness result, $\hat{X}^1$ and $\hat{X}^2$ must coincide $\hat P$-a.s.. Hence (calling $\gamma_\#$ the push-forward of the projection on the $\gamma$ component), $\gamma_\# P^{1,w,x_0}$ and $\gamma_\# P^{2,w,x_0}$, the conditional laws of $\hat X^1$ and $\hat X^2$ given $(\hat W, \hat X_0)=(w,x_0)$, coincide for $P^{W,\bar X_0}$-a.e. $(w,x_0)$. Therefore $\text{Law}(W^1,\bar X^1_0,\bar X^1) = P^{W,\bar{X}_0}\otimes \gamma_\# P^{1,w,x_0}$ and $\text{Law}(W^2,\bar X^2_0,\bar X^2) = P^{W,\bar{X}_0}\otimes \gamma_\# P^{2,w,x_0}$ coincide. The proof is complete.
\end{proof}

\begin{proposition}[Yamada-Watanabe, strong existence]\label{YamWat_exist}
The generalized solution $(\bar X^1,\bar k^1)$ is actually a strong solution to \eqref{meanfield_SDE_single}, that is, it is progressively measurable with respect to $(\mc{F}^{W^1,\bar X_0^1}_t)_t$, the filtration generated by $W^1$, $\bar X_0^1$ and the $P^1$-null sets (and similarly for $(\bar X^2,\bar k^2)$).
\end{proposition}

\begin{proof}
We continue using the notation of the previous proof. Call $(\hat{\mc{F}}^{\hat W,\hat X_0}_t)_t$ the filtration generated by $\hat W$, $\hat X_0$ and the $P^{W,\bar X_0}$-null sets on $C([0,T])\times\mr$. Note that the conditional law of $(\hat X^1,\hat k^1,\hat X^2,\hat k^2$ given $(\hat W, \hat X_0)=(w,x_0)$ is $P^{1,w,x_0} \otimes P^{2,w,x_0}$. Hence, for $P^{W,\bar X_0}$-a.e. $(w,x_0)$, conditioning to $(\hat W, \hat X_0)=(w,x_0)$, $\hat X^1$ and $\hat X^2$ coincide a.s. and are independent. Hence, for $P^{W,\bar X_0}$-a.e. $(w,x_0)$ given, conditioning to $(\hat W, \hat X_0)=(w,x_0)$, $\hat{X}^1$ must coincide with an element $Y^T(w,x_0)$ a.s.. The random element $Y^T$, extended on a $P^{W,\bar X_0}$-null set, defines a solution map $Y^T:C([0,T])\times \mr\rightarrow C([0,T])$ which is $\hat{\mc{F}}^{\hat W,\hat X_0}_T$-measurable: indeed, for every Borel subset $B$ of $C([0,T])$, $\{ Y^T\in B\}$ coincides $P^{W,\bar X_0}$-a.s. with $\{ (w,x_0) \mid \gamma_\# P^{1,w,x_0}(B) = 1\}$, which belongs to $\hat{\mc{F}}^{\hat W,\hat X_0}_T$ (since $P^{1,w,x_0}(B)$ is $\hat{\mc{F}}^{\hat W,\hat X_0}_T$-measurable). From the previous proof, we have
\begin{align*}
\text{Law}(W^1,\bar X^1_0,\bar X^1) = P^{W,\bar{X}_0}\otimes \gamma_\# P^{1,w,x_0} = \text{Law}(W,\bar X^1_0) \otimes \delta_{Y^T(w,x_0)},
\end{align*}
therefore $\bar X^1=Y^T(W^1,\bar X^1_0)$ $P^1$-a.s.. Since $(W,\bar X^1_0)$ is measurable from $\mc{F}^{W^1,\bar X_0^1}_T$ to $\hat{\mc{F}}^{\hat W,\hat X_0}_T$, we conclude that $\bar X^1$ is $\mc{F}^{W^1,\bar X_0^1}_T$-measurable.

Concerning progressive measurability, we can restrict $W^1$, $\bar X^1$ and $\bar k^1$ on $[0,t]$ and repeat the previous argument: calling $\pi^t:C([0,T])\rightarrow C([0,t])$ the restriction operator, we get that $\pi_t(\bar X^1) = Y^t(\pi_t(W^1),\bar X^1_0)$ $P^1$-a.s. and $\pi_t(\bar X^1)$ is $\mc{F}^{W^1,\bar X_0^1}_t$-measurable. Hence $\bar X$ is adapted and therefore progressively measurable, by continuity of its paths. Progressive measurability of $\bar k^1$ follows because, $P^1$-a.s.,
\begin{align*}
\rmd \bar k^1 = -\rmd \bar X^1 -\mu(\bar X^1) \rmd t +\rmd W^1 +\E^{P^1} [\mu(\bar X^1)]\rmd t +\E^{P^1} [\rmd \bar k^1].
\end{align*}
The proof is complete.
\end{proof}

\subsection{Compactness for the particle system}

Here we consider the particle system \eqref{SDE_single_reg} and we give estimates which are uniform in $N$ and $\eps$. We will often omit the superscripts $N$ and $\eps$ for notational simplicity.

\subsubsection{$BV$ estimates}

We start estimating the $BV$ norm of the average of the drift over the particles. Throughout this subsection, we will assume Conditions \ref{assumptions_mu}-(i) on $\mu$, \ref{assumptions_q} on $q$ and \ref{assumptions_x0N} on $X_0$.

\begin{lemma}\label{BV_sum}
For every $1\le p<\infty$, it holds
\begin{align*}
\sup_{N,\eps}\E\left(\frac{1}{N}\sum^N_{i=1}\int^T_0|\mu^\eps(X^{i,N,\eps}_r)|\rmd r\right)^p + \sup_{N,\eps}\E\left(\frac{1}{N}\sum^N_{i=1}\int^T_0|\rmd k^{i,N,\eps}_r|\right)^p<+\infty.
\end{align*}
\end{lemma}

The proofs of this lemma and of the next one use mainly two facts:
\begin{itemize}
\item the one-side Lipschitz property of $-\mu$ and the reflection condition on $k^i$, that is $k^i$ has the same sign of $n(X^i)$;
\item the property $\frac{1}{N}\sum^N_{i=1}X^i_t-q(t)=0$.
\end{itemize}
Let us explain briefly how these two facts yield the $BV$ estimates. We will focus only on the bounds on $k^i$, the bounds on $\mu(X^i)$ being similar. As in the standard argument for boundary terms, we take the differential of $|X^i_t-q(t)|^2$:
\begin{align}
\rmd |X^i-q|^2 = -2(X_i-q)\rmd k^i +2(X^i-q)\frac{1}{N}\sum^N_{j=1}\rmd k^j +\ldots \label{eq:BV_idea}
\end{align}
Disgarding the interaction terms, using the reflection condition, we would get an inequality of the form
\begin{align*}
\rmd |X^i-q|^2 = -2(X_i-q)\rmd k^i \le -2|X_i-q|\rmd |k^i|
\end{align*}
This inequality would give a bound on $|X_i-q(t)|\rmd |k^i|$ and so on $\rmd |k^i|$ (since $|X_i-q(t)|$ is bounded from below when $X^i$ is on the boundary). However, the interaction term $2(X^i-q)\frac{1}{N}\sum^N_{j=1}\rmd k^j$ in \eqref{eq:BV_idea} cannot be bounded as before, due to the $+$ sign instead of $-$ sign. To deal with it, firstly we average over $i$: thanks to the property $\frac{1}{N}\sum^N_{i=1}X^i_t-q(t)=0$, the average of the interaction terms disappears:
\begin{align*}
\frac{1}{N}\sum^N_{i=1}(X^i-q)\frac{1}{N}\sum^N_{j=1}\rmd k^j = 0,
\end{align*}
hence we get a bound on the average of $\rmd |k^i|$ (Lemma \ref{BV_sum}). Secondly, this bound allows to control the interaction term $2(X^i-q)\frac{1}{N}\sum^N_{j=1}\rmd k^j$. Using this control in \eqref{eq:BV_idea}, we get a bound on $\rmd |k^i|$ for any $i$ (Lemma \ref{unif_BV}).


\begin{remark}\label{rmk:drift_estimate}
The one-side Lipschitz condition on $-\mu$ and the regularity of $\mu$ in the interior $]0,1[$ imply that, for any $0<c<1/2$,
\begin{align*}
&\sup_{\eps,x\in]0,1[}(\mu^\eps(x)\text{sign}(x-1/2))^- <+\infty,\\
&\sup_{\eps,x\in[c,1-c]}|\mu^\eps(x)| <+\infty.
\end{align*}
The condition $0<\xi\le q(t)\le 1-\xi<1$ implies that
\begin{align*}
(x-q(t))\text{sign}(x-1/2)1_{x\notin[\xi/2,1-\xi/2]} \ge \frac{\xi}{2}1_{x\notin[\xi/2,1-\xi/2]}.
\end{align*}
Putting together the above bounds, we get, for some $C\ge 0$ independent of $x$ and $\eps$, for every $x$ in $]0,1[$ and every $t$,
\begin{align*}
(x-q(t))\mu^\eps(x) &= (x-q(t))\text{sign}(x-1/2)(\mu^\eps(x)\text{sign}(x-1/2))^+ 1_{x\notin[\xi/2,1-\xi/2]}\\
&\quad -(x-q(t))\text{sign}(x-1/2)(\mu^\eps(x)\text{sign}(x-1/2))^- 1_{x\notin[\xi/2,1-\xi/2]}\\
&\quad -(x-q(t))\mu^\eps(x)1_{x\in[\xi/2,1-\xi/2]}\\
&\ge \frac{\xi}{2}(\mu^\eps(x)\text{sign}(x-1/2))^+ 1_{x\notin[\xi/2,1-\xi/2]} -C\\
&= \frac{\xi}{2}|\mu^\eps(x)\text{sign}(x-1/2)|1_{x\notin[\xi/2,1-\xi/2]} -\frac{\xi}{2}(\mu^\eps(x)\text{sign}(x-1/2))^- 1_{x\notin[\xi/2,1-\xi/2]} -C \\
&\ge \frac{\xi}{2}|\mu^\eps(x)|1_{x\notin[\xi/2,1-\xi/2]} -C\\
&\ge \frac{\xi}{2}|\mu^\eps(x)| -C.
\end{align*}
By continuity of $\mu^\eps$ on $[0,1]$, for $\eps>0$ the same estimate holds on the closed interval $[0,1]$.
From the reflection condition on $k$ we also get
\begin{align*}
(X^i_t-q(t))\rmd k^i_t \ge \xi \rmd |k^i|.
\end{align*}
\end{remark}

\begin{proof}
By It\^o formula, we have
\begin{align}
&\rmd |X^i-q(t)|^2\label{eq:Ito_square}\\
&= 2(X^i-q(t))(-\mu^\eps(X^i)+\frac{1}{N}\sum^N_{j=1}\mu^\eps(X^j))\rmd t +2\sigma(X^i-q(t))(\rmd W^i-\frac{1}{N}\sum^N_{j=1}\rmd W^j)\nonumber\\
&\ \ +\sigma^2(1-\frac{1}{N})\rmd t +2(X^i-q(t))(-\rmd k^i+\frac{1}{N}\sum^N_{j=1}\rmd k^j).\nonumber
\end{align}
We average over $i$. For the interaction term with $\frac{1}{N}\sum^N_{j=1}\mu^\eps(X^j)$, the condition $\frac{1}{N}\sum^N_{i=1}X^i_t=q(t)$ implies
\begin{align*}
\frac{1}{N}\sum^N_{i=1}(X^i-q(t))\frac{1}{N}\sum^N_{j=1}\mu^\eps(X^j)\rmd t = 0
\end{align*}
and similarly for the other interaction terms (with $\frac{1}{N}\sum^N_{j=1}\rmd W^j$ and with $\frac{1}{N}\sum^N_{j=1}\rmd k^j$). Hence we get
\begin{align*}
&\rmd \frac{1}{N}\sum^N_{i=1}|X^i-q(t)|^2\\
&= -2\frac{1}{N}\sum^N_{i=1}(X^i-q(t))\mu^\eps(X^i)\rmd t +2\frac{1}{N}\sum^N_{i=1}\sigma(X^i-q(t))\rmd W^i\\
&\ \ +\sigma^2(1-\frac{1}{N})\rmd t -2\frac{1}{N}\sum^N_{i=1}(X^i-q(t))\rmd k^i.
\end{align*}
Now we apply Remark \ref{rmk:drift_estimate} and obtain
\begin{align*}
&\frac{1}{N}\sum^N_{i=1}|X^i_T-q(T)|^2 + \xi\frac{1}{N}\sum^N_{i=1}\int^T_0|\mu^\eps(X^i)|\rmd r +2\xi\frac{1}{N}\sum^N_{i=1}\int^T_0\rmd |k^i|_r\\
&\le \frac{1}{N}\sum^N_{i=1}|X^i_T-q(T)|^2 +2\frac{1}{N}\sum^N_{i=1}\int^T_0(X^i-q(t))\mu^\eps(X^i)\rmd t +CT +2\frac{1}{N}\sum^N_{i=1}\int^T_0(X^i-q(t))\rmd k^i\\
&\le \frac{1}{N}\sum^N_{i=1}|X^i_0-q(0)|^2 + CT +2\sigma\left|\frac{1}{N}\sum^N_{i=1}\int^T_0(X^i_r-q(r))\rmd W^i_r\right| +\sigma^2T\\
&\le C +2\sigma\left|\frac{1}{N}\sum^N_{i=1}\int^T_0(X^i_r-q(r))\rmd W^i_r\right|
\end{align*}
By Burkholder-Davis-Gundy inequality (and boundedness of $X$ and $q$), we arrive at
\begin{align*}
&C'\E\left(\frac{1}{N}\sum^N_{i=1}\int^T_0|\mu^\eps(X^i)|\rmd r\right)^p +C'\E\left(\frac{1}{N}\sum^N_{i=1}\int^T_0\rmd |k^i|_r\right)^p\\
&\le C +C\E\left(\frac{1}{N}\sum^N_{i=1}\int^T_0|X^i_r-q(r)|^2\rmd r\right)^{p/2}\le C,
\end{align*}
which is the desired bound.
\end{proof}

Thanks to the previous Lemma, we can conclude a uniform $BV$ estimate on the drift. 

\begin{lemma}\label{unif_BV}
For every $1\le p<\infty$, it holds
\begin{align*}
\sup_{N,\eps,i=1\ldots N}\E\left(\int^T_0|\mu^\eps(X^{\eps,N,i})|\rmd r\right)^p +\sup_{N,\eps,i=1\ldots N}\E\left(\int^T_0|\rmd k^{N,\eps,i}_r|\right)^p <+\infty.
\end{align*}
\end{lemma}

\begin{proof}
We start as before from formula \eqref{eq:Ito_square}, for fixed $i$, and use Remark \ref{rmk:drift_estimate}, getting
\begin{align*}
&|X^i_T-q(T)|^2 + \xi\int^T_0|\mu^\eps(X^i)|\rmd r +2\xi\int^T_0\rmd |k^i|_r\\
&\le |X^i_0-q(0)|^2 + CT +2\int^T_0|X^i_r-q(r)|\frac{1}{N}\sum^N_{j=1}|\mu^\eps(X^j_r)|\rmd r\\
&\ \ +2\sigma\left|\int^T_0(X^i_r-q(r))(\rmd W^i-\frac{1}{N}\sum^N_{j=1}\rmd W^j)\right| +\sigma^2T +2\int^T_0|X^i_r-q(r)|\frac{1}{N}\sum^N_{j=1}\rmd k^j_r\\
&\le C +2\int^T_0\frac{1}{N}\sum^N_{j=1}|\mu^\eps(X^j_r)|\rmd r +2\sigma\left|\int^T_0(X^i_r-q(r))(\rmd W^i-\frac{1}{N}\sum^N_{j=1}\rmd W^j)\right| +2\int^T_0\frac{1}{N}\sum^N_{j=1}\rmd k^j_r.
\end{align*}
By Burkholder-Davis-Gundy inequality, we get
\begin{align*}
&C'\E\left(\int^T_0|\mu^\eps(X^i)|\rmd r\right)^p +C'\E\left(\int^T_0\rmd |k^i|_r\right)^p\\
&\le C +C\E\left(\int^T_0\frac{1}{N}\sum^N_{j=1}|\mu^\eps(X^j_r)|\rmd r\right)^p +C +C\E\left(\int^T_0\frac{1}{N}\sum^N_{j=1}\rmd k^j_r\right)^p.
\end{align*}
Here we use Lemma \ref{BV_sum} and conclude
\begin{align*}
&C'\E\left(\int^T_0|\mu^\eps(X^i)|\rmd r\right)^p +C'\E\left(\int^T_0\rmd |k^i|_r\right)^p\le C.
\end{align*}
The proof is complete.
\end{proof}

\subsubsection{H\"older estimates}

In this paragraph we use a similar strategy to estimate the H\"older norm of $X^i$, first taking the average over $i$ to remove the interaction term, then using the estimate on the average to control the interaction term. In order to bound the H\"older norm, we take the It\^o differential of $|X^i_t-q(t)-X^i_s+q(s)|^2$ (instead of just $|X^i_t-q(t)|^2$).

Throughout this subsection, we will assume Conditions \ref{assumptions_mu}-(i,ii) on $\mu$, \ref{assumptions_q} on $q$ and \ref{assumptions_x0N} on $X_0$.

We start with a preliminary result which will be used in the next estimates:

\begin{lemma}\label{Holder_boundary}
For every $1\le p<\infty$, it holds for some $C_p\ge0$ independent of $s,t$,
\begin{align*}
\sup_{N,\eps,i}\E\left(\int^t_s[(X^i_r-q(r)-X^i_s+q(s))\mu^\eps(X^i_r)]^-\rmd r\right)^p\le C_p|t-s|^p,\\
\sup_{N,\eps,i}\E\left(\int^t_s[(X^i_r-q(r)-X^i_s+q(s))n(X^i_r)]^-|\rmd k^i_r|\right)^p\le C_p|t-s|^p.
\end{align*}
\end{lemma}

\begin{proof}
We start with the first inequality and we fix $\delta>0$ small, independently of $\eps$ and $N$. Using the elementary inequality $[a+b]^- \le |a|+[b]^-$ for $a,b\in \mathbb{R}$, we have
\begin{align*}
&[(X^i_r-q(r)-X^i_s+q(s))\mu^\eps(X^i_r)]^-\\
&\le |q(r)-q(s)||\mu^\eps(X^i_r)|+[(X^i_r-X^i_s)\mu^\eps(X^i_r)]^-\\
&\le |q(r)-q(s)||\mu^\eps(X^i_r)|+|X^i_r-X^i_s|\max_{[\delta,1-\delta]}|\mu^\eps|\\
&\ \ +(1_{X^i_r<\delta,X^i_r\le X^i_s}+1_{1-\delta<X^i_r,X^i_s\le X^i_r})|X^i_r-X^i_s|[\text{sign}(X^i_r-1/2)\mu^\eps(X^i_r)]^-\\
&\ \ +(1_{X^i_s<X^i_r<\delta}+1_{1-\delta<X^i_r<X^i_s})|X^i_r-X^i_s||\mu^\eps(X^i_r)|.
\end{align*}
For the first added in the RHS, the Lipschitz property of $q$ and Lemma \ref{unif_BV} give
\begin{align*}
\E\left(\int^t_s|q(r)-q(s)||\mu^\eps(X^i_r)|\rmd r\right)^p\le C|t-s|^p\E\left(\int^T_0|\mu^\eps(X^i_r)|\rmd r\right)^p \le C_p|t-s|^p.
\end{align*}
For the second and third addends, we have by Remark \ref{rmk:drift_estimate} (recall $\delta$ is fixed and $X^i$ is in $[0,1]$)
\begin{align*}
&\E\left(\int^t_s|X^i_r-X^i_s|\max_{[\delta,1-\delta]}|\mu^\eps|\rmd r\right)^p \le C_p|t-s|^p,\\
&\E\left(\int^t_s (1_{X^i_r<\delta,X^i_r\le X^i_s}+1_{1-\delta<X^i_r,X^i_s\le X^i_r}) |X^i_r-X^i_s|[\text{sign}(X^i_r-1/2)\mu^\eps(X^i_r)]^-\rmd r\right)^p\le C_p|t-s|^p.
\end{align*}
Concerning the fourth addend, we consider only the case $1-\delta<X^i_r<X^i_s$, the case $X^i_s<X^i_r<\delta$ being completely analogous. By the assumption \ref{assumptions_mu}-(ii) we have
\begin{align*}
1_{1-\delta<X^i_r<X^i_s}|X^i_r-X^i_s||\mu^\eps(X^i_r)|\le C1_{1-\delta<X^i_r<X^i_s}\sup_{1-\delta<x<1}(1-x)|\mu(x)|  \le C,
\end{align*}
Therefore, reasoning similarly for $X^i_s<X^i_r<\delta$, we have
\begin{align*}
\E\left(\int^t_s(1_{1-\delta<X^i_r<X^i_s}+1_{1-\delta<X^i_r<X^i_s})|X^i_r-X^i_s||\mu^\eps(X^i_r)|\rmd r\right)^p\le C_p|t-s|^p.
\end{align*}
Putting all together, we obtain the first estimate.

For the second estimate, recall that $(X^i_r-X^i_s)n(X^i_r)1_{X^i_r\in \partial]0,1[}\ge 0$. Therefore
\begin{align*}
&[(X^i_r-q(r)-X^i_s+q(s))n(X^i_r)]^-|\rmd k^i_r|\\
&= [(X^i_r-q(r)-X^i_s+q(s))n(X^i_r)]^-1_{X^i_r\in\partial]0,1[}|\rmd k^i_r|\\
&\le |q(r)-q(s)||\rmd k^i_r|.
\end{align*}
The Lipschitz property of $q$ and Lemma \ref{unif_BV} give
\begin{align*}
&\E\left(\int^t_s|q(r)-q(s)||\rmd k^i_r|\right)^p\le C|t-s|^p\E\left(\int^T_0|\rmd k^i_r|\right)^p \le C_p|t-s|^p
\end{align*}
and we arrive at the second estimate.
\end{proof}

\begin{remark}\label{rmk:tech_1}
Only in the above proof we use Condition \ref{assumptions_mu}-(ii). If $\mu$ diverged at the boundaries like $x^{-\alpha}$ for some $\alpha>1$, then a similar result to Lemma \ref{Holder_boundary} should hold, but with $(X^i_r-q(r)-X^i_s+q(s))^\alpha$ in place of $(X^i_r-q(r)-X^i_s+q(s))$. However, such result would not be enough, since in the next Lemma \ref{Holder_sum} the power-$1$ factor $(X^i_r-q(r)-X^i_s+q(s))$ appears and is needed to cancel the interaction term when taking the average. We also expect, for $\mu$ diverging like $x^{-\alpha}$ with $\alpha>1$, that the particles should not even touch the boundaries (as it is without interaction), but we do not focus on this point.
\end{remark}

We estimate the H\"older norm of the average of the drift over $i$:

\begin{lemma}\label{Holder_sum}
There exists $0<\alpha\le 1/2$ such that, for every $1\le p<\infty$, it holds, for some $C_p\ge0$ independent of $s,t$,
\begin{align*}
\sup_{N,\eps}\E\left(\frac{1}{N}\sum^N_{i=1}\int^t_s|\mu^\eps(X^{\eps,N,i})|\rmd r\right)^p +\sup_{N,\eps}\E\left(\frac{1}{N}\sum^N_{i=1}\int^t_s|\rmd k^{\eps,N,i}_r|\right)^p\le C_p|t-s|^{\alpha p}.
\end{align*}
\end{lemma}

\begin{proof}
We start estimating the H\"older norm of $\frac{1}{N}\sum^N_{i=1}|X^i-q|^2$. For this we fix $s$ and we have, for $t>s$,
\begin{align}
&\rmd |X^i_t-q(t)-X^i_s+q(s)|^2\label{Holder_eq}\\
&= 2(X^i_t-q(t)-X^i_s+q(s))(-\mu^\eps(X^i_t)+\frac{1}{N}\sum^N_{i=1}\mu^\eps(X^j_t))\rmd t\nonumber\\
&\ \ +2\sigma(X^i_t-q(t)-X^i_s+q(s))(\rmd W^i_t-\frac{1}{N}\sum^N_{j=1}\rmd W^j_t) +\sigma^2(1-\frac{1}{N})\rmd t\nonumber\\
&\ \ +2(X^i_t-q(t)-X^i_s+q(s))(-\rmd k^i_t+\frac{1}{N}\sum^N_{j=1}\rmd k^j_t).\nonumber
\end{align}
Similarly to the argument in Lemma \ref{BV_sum}, averaging over $i$ we get rid of the interaction terms $\frac{1}{N}\sum^N_{i=1}\mu^\eps(X^j_t)$, $\frac{1}{N}\sum^N_{j=1}\rmd W^j_t$ and $\frac{1}{N}\sum^N_{j=1}\rmd k^j_t$:
\begin{align*}
&\rmd \frac{1}{N}\sum^N_{i=1}|X^i_t-q(t)-X^i_s+q(s)|^2\\
&= -2\frac{1}{N}\sum^N_{i=1}(X^i_t-q(t)-X^i_s+q(s))\mu^\eps(X^i_t)\rmd t\\
&\ \ +2\sigma\frac{1}{N}\sum^N_{i=1}(X^i_t-q(t)-X^i_s+q(s))\rmd W^i_t +\sigma^2(1-\frac{1}{N})\rmd t -2(X^i_t-q(t)-X^i_s+q(s))\rmd k^i_t.
\end{align*}
We take the $p$-power and obtain
\begin{align*}
&\left(\frac{1}{N}\sum^N_{i=1}|X^i_t-q(t)-X^i_s+q(s)|^2\right)^p\\
&\le C_p\left(\frac{1}{N}\sum^N_{i=1}\int^t_s[(X^i_r-q(r)-X^i_s+q(s))\mu^\eps(X^i_r)]^-\rmd r\right)^p\\
&\ \ +C_p\sigma^p\left|\frac{1}{N}\sum^N_{i=1}\int^t_s(X^i_r-q(r)-X^i_s+q(s))\rmd W^i_r\right|^p +C_p\sigma^{2p}|t-s|^p\\
&\ \ +C_p\left(\frac{1}{N}\sum^N_{i=1}\int^t_s[(X^i_r-q(r)-X^i_s+q(s))n(X^i_r)]^-|\rmd k^i_r|\right)^p.
\end{align*}
The first addend of the RHS is controlled via Lemma \ref{Holder_boundary} and Jensen inequality (applied to the average over $i$):
\begin{align*}
&\E\left(\frac{1}{N}\sum^N_{i=1}\int^t_s[(X^i_r-q(r)-X^i_s+q(s))\mu^\eps(X^i_r)]^-\rmd r\right)^p\\
&\le \sup_{N,\eps,i}\E\left(\int^t_s[(X^i_r-q(r)-X^i_s+q(s))\mu^\eps(X^i_r)]^-\rmd r\right)^p \le C_p|t-s|^{p}.
\end{align*}
Similarly for the last addend. The second addend is controlled via Burkholder-Davis-Gundy inequality and Jensen inequality:
\begin{align*}
&\E\left|\frac{1}{N}\sum^N_{i=1}\int^t_s(X^i_r-q(r)-X^i_s+q(s))\rmd W^i_r\right|^p\\
&\le \sup_{N,\eps,i}\E\left|\int^t_s(X^i_r-q(r)-X^i_s+q(s))\rmd W^i_r\right|^p \le C_p\sup_{N,\eps,i}\E\left(\int^t_s|X^i_r-q(r)-X^i_s+q(s)|^2\rmd r\right)^{p/2}\\
&\le C_p\sup_{N,\eps,i}\E\sup_t|X^i_t-q(t)|^p|t-s|^{p/2} \le C_p|t-s|^{p/2}
\end{align*}
Therefore we have
\begin{align*}
\E\left(\frac{1}{N}\sum^N_{i=1}|X^i_t-q(t)-X^i_s+q(s)|^2\right)^p\le C_p|t-s|^{p/2}.
\end{align*}
Now we recall the following elementary inequality (consequence of Cauchy-Schwarz inequality), for every two sequences of real numbers $a_i$, $b_i$:
\begin{align*}
\left|\frac{1}{N}\sum^N_{i=1}(a_i^2-b_i^2)\right| \le \left(\frac{1}{N}\sum^N_{i=1}|a_i-b_i|^2\right)^{1/2}\left(\frac{1}{N}\sum^N_{i=1}|a_i+b_i|^2\right)^{1/2}.
\end{align*}
Applying this inequality to $a_i=X^i_t-q(t)$, $b_i=X^i_s-q(s)$ and using Jensen inequality, we get the H\"older bound on $\frac{1}{N}\sum^N_{i=1}|X^i-q|^2$:
\begin{align}
&\E\left|\frac{1}{N}\sum^N_{i=1}|X^i_t-q(t)|^2-|X^i_s-q(s)|^2\right|^p\nonumber\\
&\le \E\left[\left(\frac{1}{N}\sum^N_{i=1}|X^i_t-q(t)-X^i_s+q(s)|^2\right)^{p/2}\left(\frac{1}{N}\sum^N_{i=1}|X^i_t-q(t)+X^i_s-q(s)|^2\right)^{p/2}\right]\nonumber\\
&\le \left(\E\left(\frac{1}{N}\sum^N_{i=1}|X^i_t-q(t)-X^i_s+q(s)|^2\right)^p\right)^{1/2}\left(\E\frac{1}{N}\sum^N_{i=1}|X^i_t-q(t)+X^i_s-q(s)|^{2p}\right)^{1/2}\nonumber\\
&\le C_p|t-s|^{p/4} \sup_{N,\eps,i}(\E\sup_t|X^i_t-q(t)|^{2p})^{1/2}\le C_p|t-s|^{p/4}.\label{Holder_averX}
\end{align}
On the other hand, averaging \eqref{eq:Ito_square} and using again the cancellation of the interaction terms, we get the equation for $\frac{1}{N}\sum^N_{i=1}|X^i-q|^2$:
\begin{align*}
&\frac{1}{N}\sum^N_{i=1}|X^i_t-q(t)|^2-|X^i_s+q(s)|^2\\
&= -2\frac{1}{N}\sum^N_{i=1}\int^t_s(X^i_r-q(r))\mu^\eps(X^i_r)\rmd r +2\frac{1}{N}\sum^N_{i=1}\int^t_s\sigma(X^i_r-q(r))\rmd W^i_r\\
&\ \ +\sigma^2(1-\frac{1}{N})(t-s) -2\frac{1}{N}\sum^N_{i=1}\int^t_s(X^i_r-q(r))\rmd k^i_r,
\end{align*}
and so, by Remark \ref{rmk:drift_estimate}, we obtain
\begin{align*}
&\E\left(\frac{1}{N}\sum^N_{i=1}\int^t_s|\mu^\eps(X^i_r)|\rmd r\right)^p +\E\left(\frac{1}{N}\sum^N_{i=1}\int^t_s|\rmd k^i_r|\right)^p\\
&\le C_p \E\left(\frac{1}{N}\sum^N_{i=1}\int^t_s(X^i_r-q(r))\mu^\eps(X^i_r)\rmd r\right)^p +\E\left(\frac{1}{N}\sum^N_{i=1}\int^t_s(X^i_r-q(r))\rmd k^i_r\right)^p +C_p(t-s)^p\\
&\le C_p\left|\frac{1}{N}\sum^N_{i=1}|X^i_t-q(t)|^2-|X^i_s+q(s)|^2\right|^p +C_p\left|\frac{1}{N}\sum^N_{i=1}\int^t_s\sigma(X^i_r-q(r))\rmd W^i_r\right|^p +C_p(t-s)^p.
\end{align*}
We control the first addend in the RHS by \eqref{Holder_averX} and the second addend by Burkholder-Davis-Gundy inequality (and Jensen inequality on the average over $i$):
\begin{align*}
&\E\left(\frac{1}{N}\sum^N_{i=1}\int^t_s|\mu^\eps(X^i_r)|\rmd r\right)^p +\E\left(\frac{1}{N}\sum^N_{i=1}\int^t_s|\rmd k^i_r|\right)^p\\
&\le C_p|t-s|^{p/4} + C_p\sup_{N,\eps,i}(\E\sup_t|X^i_r-q(t)|^p)|t-s|^{p/2} +C_p(t-s)^p \le C_p|t-s|^{p/4},
\end{align*}
which is the desired estimate with $\alpha =1/4$.
\end{proof}

Now we can prove the uniform H\"older bound:

\begin{lemma}\label{uniform_Holder}
With the notation of the previous Lemma, for every $1\le p<\infty$, it holds, for some $C_p\ge0$ independent of $s,t$,
\begin{align*}
\sup_{N,\eps,i}\E|X^i_t-X^i_s|^p\le C_p|t-s|^{\alpha p/2}.
\end{align*}
\end{lemma}

\begin{proof}
By Jensen inequality, it is enough to prove the estimate for $p\ge 2$. We start again with the equation \eqref{Holder_eq} for a fixed $i$. Taking the $p/2$-power we obtain
\begin{align*}
&|X^i_t-q(t)-X^i_s+q(s)|^p\\
&\le C_p\left(\int^t_s[(X^i_r-q(r)-X^i_s+q(s))\mu^\eps(X^i_r)]^-\rmd r\right)^{p/2}\\
&\ \ +C_p\left(\int^t_s|X^i_r-q(r)-X^i_s+q(s)|\frac{1}{N}\sum^N_{j=1}|\mu^\eps(X^j_t)|\rmd r\right)^{p/2}\\
&\ \ +C_p\sigma^p\left|\int^t_s(X^i_r-q(r)-X^i_s+q(s))\rmd(W^i_r-\frac{1}{N}\sum^N_{j=1}W^j_r)\right|^{p/2} +C_p\sigma^{2p}|t-s|^{p/2}\\
&\ \ +C_p\left(\int^t_s[(X^i_r-q(r)-X^i_s+q(s))n(X^i_r)]^-|\rmd k^i_r|\right)^{p/2}\\
&\ \ +C_p\left(\int^t_s|X^i_r-q(r)-X^i_s+q(s)|\frac{1}{N}\sum^N_{j=1}|\rmd k^j_r|\right)^{p/2}
\end{align*}
The first addend of the RHS is controlled again via Lemma \ref{Holder_boundary}:
\begin{align*}
\E\left(\int^t_s[(X^i_r-q(r)-X^i_s+q(s))\mu^\eps(X^i_r)]^-\rmd r\right)^{p/2}\le C_p|t-s|^{p/2}.
\end{align*}
Similarly for the fourth addend. The previous Lemma \ref{Holder_sum} allows to control the second addend:
\begin{align*}
&\E\left(\int^t_s|X^i_r-q(r)-X^i_s+q(s)|\frac{1}{N}\sum^N_{j=1}|\mu^\eps(X^j_t)|\rmd r\right)^{p/2}\\
&\le C_p\E\left(\int^t_s\frac{1}{N}\sum^N_{j=1}|\mu^\eps(X^j_t)|\rmd r\right)^{p/2}\le C_p|t-s|^{\alpha p/2}
\end{align*}
Similarly for the fifth addend. The third addend is controlled via Burkholder-Davis-Gundy inequality:
\begin{align*}
&\E\left(\int^t_s(X^i_r-q(r)-X^i_s+q(s))\rmd(W^i_r-\frac{1}{N}\sum^N_{j=1}W^j_r)\right)^{p/2}\\
&\le C_p\sup_{N,\eps,i}\E\left(\int^t_s|X^j_r-q(r)-X^i_s+q(s)|^2\rmd r\right)^{p/4}\le C_p|t-s|^{p/4}
\end{align*}
Putting all together we get
\begin{align*}
\E|X^i_t-q(t)-X^i_s+q(s)|^p\le C_p|t-s|^{\alpha p/2}.
\end{align*}
Using the Lipschitz continuity of $q$, we obtain the desired bound.
\end{proof}

\begin{remark}
We have shown that the H\"older exponent is $\alpha/2=1/8$ (since we can take $\alpha=1/4$). This is a consequence of our argument, but we expect that the optimal H\"older exponent is still $1/2$. 
\end{remark}

We conclude with a H\"older estimate on the total variation of the drift:

\begin{lemma}\label{unif_Holder_drift}
For every $1\le p<\infty$, it holds for some $C_p\ge0$ independent of $s,t$
\begin{align*}
\sup_{N,\eps,i}\E\left(\int^t_s|\mu^\eps(X^{\eps,N,i}_r)|\rmd r\right)^p +\sup_{N,\eps,i}\E\left(\int^t_s|\rmd k^{\eps,N,i}_r|\right)^p \le C |t-s|^{\alpha p/2}.
\end{align*}
\end{lemma}

\begin{proof}
The equation \eqref{eq:Ito_square} for $|X_t-q(t)|^2$ implies
\begin{align*}
&\int^t_s 2(X^i-q(r))\mu^\eps(X^i)\rmd r +\int^t_s 2(X^i-q(r))\rmd k^i_r\\
&\le |X^i_s-q(s)|^2-|X^i_t-q(t)|^2 +C\int^t_s\frac{1}{N}\sum^N_{j=1}|\mu^\eps(X^j)|\rmd r\\
&\ \ +\left|\int^t_s2\sigma(X^i-q(t))(\rmd W^i-\frac{1}{N}\sum^N_{j=1}\rmd W^j)\right| +\sigma^2 (t-s) +C\int^t_s\frac{1}{N}\sum^N_{i=1}|\rmd k^j_r|.
\end{align*}
and so, by Remark \ref{rmk:drift_estimate},
\begin{align*}
&\int^t_s|\mu^\eps(X^i)|\rmd r +\int^t_s|\rmd k^{\eps,N,i}_r|\\
&\le C|X^i_s-X^i_t-q(s)+q(t)||X^i_s+X^i_t-q(s)-q(t)| +C\int^t_s\frac{1}{N}\sum^N_{j=1}|\mu^\eps(X^j)|\rmd r\\
&\ \ +C\left|\int^t_s\sigma(X^i-q(t))(\rmd W^i-\frac{1}{N}\sum^N_{j=1}\rmd W^j)\right| +C\sigma^2 (t-s) +C\int^t_s\frac{1}{N}\sum^N_{j=1}|\rmd k^j_r|.
\end{align*}
By Burkholder-Davis-Gundy inequality, we get
\begin{align*}
&\E\left(\int^t_s|\mu^\eps(X^i)|\rmd r\right)^p +\E\left(\int^t_s|\rmd k^{\eps,N,i}_r|\right)^p\\
&\le C\E|X^i_s-X^i_t-q(s)+q(t)|^{p} +C\E\left(\int^t_s\frac{1}{N}\sum^N_{j=1}|\mu^\eps(X^j)|\rmd r\right)^p\\
&\ \ +C(t-s)^p +C\E\left(\int^t_s\frac{1}{N}\sum^N_{j=1}|\rmd k^j_r|\right)^p.
\end{align*}
Lemma \ref{Holder_sum} and Lemma \ref{uniform_Holder} allow to conclude the desired bound.
\end{proof}

\subsection{Convergence of the particle system}

In this Subsection we show the convergence of the regularized particle system \eqref{SDE_single_reg} to the McKean-Vlasov SDE \eqref{meanfield_SDE_single}, by a compactness argument; as a consequence, we get the existence of a solution to \eqref{meanfield_SDE_single}. We are given a probability space $(\Omega,\mc{A},P)$ and independent Brownian motions $W^i$, $i\ge 1$, on a (right-continuous complete) filtration $(\mc{F}_t)_t$. For each $N$, we are given $(X^{1,N}_0,\ldots X^{N,N}_0)$ $\mc{F}_0$-measurable random variable and we let $(X^{i,N,\eps},k^{i,N,\eps})$ be the corresponding solution to the regularized $N$-particle system \eqref{SDE_single_reg}. Through all the section, we assume Conditions \ref{assumptions_mu}, \ref{assumptions_q} and \ref{assumptions_x0N}.

%

In the following we use the notation $C_t=C([0,T];\mr^d)$, $C_{t,[0,1]}=C([0,T];[0,1])$; we use also $W_t^{\beta,p}=W^{\beta,p}([0,T])$ for the fractional Sobolev space of order $0<\beta<1$ and exponent $1\le p<\infty$, with norm
\begin{align*}
\|f\|_{W_t^{\beta,p}}^p = \int^T_0 |f(t)|^p\rmd t + \int^T_0\int^T_0\frac{|f(t)-f(s)|^p}{|t-s|^{1+\beta p}}\rmd s\rmd t.
\end{align*}
We consider the Polish space $E=C_t\times C_{t,[0,1]}\times C_t$, endowed with its Borel $\sigma$-algebra. We denote a generic element of $E$ as $\gamma=(\gamma^1,\gamma^2,\gamma^3)$ or (for reasons that will be clear later) $(W,X,Z)$; with a little abuse of notation, we use
\begin{align*}
W,X,Z
\end{align*}
also to denote the canonical projections on $E$. We consider also the space $\mc{P}(E)$ of probability measures on $E$ with the topology of weak convergence of probability measure (also endowed with its Borel $\sigma$-algebra). The space $\mc{P}(E)$ is a Polish space as well \cite[Remark 7.1.7]{ambrosio2008gradient}. For a measure $\nu$ and a function $g$ on $E$, we use the notation $\nu(g) = \int_E g\rmd \mu$ (when the integral makes sense).

We take the random empirical measures on $E$ given by
\begin{align*}
L^{N,\eps} = \frac{1}{N}\sum^N_{i=1} \delta_{(W^i(\omega),X^{i,N,\eps}(\omega),-\int^\cdot_0\mu^\eps(X^{i,N,\eps}_r)\rmd r-k^{i,N,\eps})},
\end{align*}
which are random variables on $\mc{P}(E)$. Note that, for $\omega$ in $\Omega$, for any Borel bounded or non-negative function $g:E\to \mr$,
\begin{align*}
\E^{L^{N,\eps}(\omega)}[g(W,X,Z)] = \frac{1}{N}\sum_{i=1}^N g\left(W^i(\omega),X^{i,N,\eps}(\omega),-\int^\cdot_0\mu^\eps(X^{i,N,\eps}_r)\rmd r-k^{i,N,\eps}\right),
\end{align*}
where $\E^{L^{N,\eps}}$ is the expectation under $L^{N,\eps}$. By equation \eqref{SDE_single_reg} and the definition of $L^{\eps,N}(\omega)$, for $P$-a.e.~$\omega$, under the measure $L^{N,\eps}(\omega)$ it holds on $E$: for every $t$,
\begin{align}
\begin{aligned}\label{eq:particle_equiv}
&X_t= X_0 +q(t)-q(0) +\sigma W_t +Z_t -\sigma\E^{L^{N,\eps}}[W_t] -\E^{L^{N,\eps}}[Z_t],\\
&Z_t= -\int_0^t \mu^\eps(X_r)\rmd r -k_t,\\
&\rmd|k|_t = 1_{X_t\in \{0,1\}}\rmd|k|_t,\ \ \rmd k_t = n(X_t)\rmd|k|_t.
\end{aligned}
\end{align}


\begin{proposition}
Assume Conditions \ref{assumptions_mu}, \ref{assumptions_q} and \ref{assumptions_x0N} (actually, Condition \ref{assumptions_mu}-(iii) is not needed). Then the family $(\text{Law}(L^{N,\eps}))_{N,\eps}$ (probability measures on $\mc{P}(E)$) is tight.
\end{proposition}


\begin{remark}\label{rmk_tightness}
In view of the proof, we recall the following standard/known facts:
\begin{itemize}
\item To prove that a family of probability measures $(P^n)_n$ on a metric space $\chi$ is tight, it is enough to find a nonnegative function $F$ on $\chi$, such that $F$ is coercive (that is, with compact sublevel sets) and $\int_\chi F \rmd P^n$ is bounded uniformly in $n$.
\item When $\chi=\mc{P}(E)$ for $E$ as above (and more generally for every Polish space $E$), endowed with the topology of weak convergence, we can take $F(\nu) = \int_E g\rmd \nu$ as nonnegative coercive function on $\mc{P}(E)$, provided that $g:E\rightarrow \mr$ is a nonnegative coercive function on $E$. Indeed, every sublevel set $\{F\le C\}$ is compact: for any sequence $(\nu^n)_n$ of measures on $E$, if all $\nu^n$ belong to $\{F\le C\}$, then, by the previous point applied to $g$, $(\nu^n)_n$ is tight, hence there exists a subsequence which is weakly convergent to some measure $\nu$ on $E$, and $\mu$ also belongs to $\{F\le C\}$ by Fatou lemma.
\item By Sobolev embedding, there exists $C>0$ such that, for every $\gamma$ in $E=C_t\times C_{t,[0,1]}\times C_t$,
\begin{align*}
\|\gamma\|_{C^\alpha_t} \le C\|\gamma\|_{W^{\beta,p}_t}
\end{align*}
for $\alpha = \beta - 1/p$, provided that $\beta - 1/p>0$. By Ascoli-Arzel\`a theorem, the norm $\|\cdot\|_{C^\alpha_t}$ is coercive on $E$ for $\alpha>0$. Therefore, to show that a certain family of probability measures $P^n$ on $\mc{P}(E)$ is tight, it is enough to show that
\begin{align*}
\sup_n\E^{P^n} \int_{E} \|\gamma\|_{W^{\beta,p}_t} \mu(\rmd \gamma)<\infty
\end{align*}
for some $\beta>0$, $p\ge1$ with $\beta - 1/p>0$ (here $\E^{P^n}$ denotes the expectation under $P^n$).
\end{itemize}
\end{remark}

\begin{proof}
By the previous Remark \ref{rmk_tightness}, it is enough to verify that, for some $\beta>0$, $p\ge 1$ with $\beta>1/p$, for $h=1,2,3$, 
\begin{align*}
\sup_n\E \int_{E} \|\gamma^h\|_{W^{N,\eps}_t} L^{\eps,N}(\rmd \gamma)<\infty.
\end{align*}

For $h=1$, that is the Brownian motion component, we have
\begin{align*}
\E \int_{E} \|\gamma^1\|_{W^{\beta,p}_t} L^{N,\eps}(\rmd \gamma) = \E \frac{1}{N}\sum^N_{i=1} \|W^i\|_{W^{\beta,p}_t} = \E \|W^1\|_{W^{\beta,p}_t} <\infty
\end{align*}
for any $\beta<1/2$ and $p\ge 1$.

For $h=2$, that is the solution component, by Lemma \ref{uniform_Holder} we get, for some $\beta>0$, for every $p\ge 1$ and $0<\delta<\beta$, for every $i=1,\ldots N$,
\begin{align*}
&\E [\|X^{i,N,\eps}\|_{W^{\beta-\delta,p}_t}^p]\\
&= \int^T_0 \E |X^{i,N,\eps}_t|^p \rmd t +\int^T_0\int^T_0 \frac{\E |X^{i,N,\eps}_t-X^{i,N,\eps}_s|^p}{|t-s|^{1+(\beta-\delta) p}} \rmd s\rmd t\\
&\le T +\int^T_0\int^T_0 |t-s|^{-(1-\delta p)}\rmd s\rmd t\, \sup_{s,t}\frac{\E |X^{i,N,\eps}_t-X^{i,N,\eps}_s|^p}{|t-s|^{\beta p}} \le C
\end{align*}
for some $C>0$ independent of $N$ and $\eps$. It follows that
\begin{align*}
\E \|\gamma^2\|_{W^{\beta-\delta\delta,p}_t} = \E \frac{1}{N}\sum^N_{i=1} \|X^{i,N,\eps}\|_{W^{\beta-\delta,p}_t} \le C.
\end{align*}

A similar argument, using Lemma \ref{unif_Holder_drift} in place of Lemma \ref{uniform_Holder}, works for $h=3$. The proof is complete.
\end{proof}

%
%
%

From now on, we assume that $\frac{1}{N}\sum_{i=1}^N\delta_{X^{i,N}_0}$ converges in law to some probability measure $\text{Law}(X_0)$.

In the following, we fix a limit point $Q$ of $\text{Law}(L^{N,\eps})$ and a $\mc{P}(E)$-valued random variable $L$ with law $Q$. With a little abuse of notation, we do not re-label the subsequence of $\text{Law}(L^{N,\eps})$ converging to $Q$, and we assume that $L$ is also defined on the same probability space $(\Omega,\mc{A},P)$ (this is only for notational simplicity). We call $Q^{N,\eps,e}$, $Q^e$ the (deterministic) probability measures on $E$ obtained averaging resp.~$L^{N,\eps}$, $L$, namely, for every Borel bounded or nonnegative function $g$ on $E$,
\begin{align*}
\E^{Q^{\eps,N,e}}[g] = \E[L^{\eps,N}(g)], \quad \E^{Q^e}[g] = \E[L(g)],
\end{align*}
where $\E^Q$ denotes the expectation with respect to $Q$.

\begin{remark}\label{rmk:conv_random_meas}
We recall some useful facts of convergence in law of random probability measures.
\begin{itemize}
\item Let $H:E\rightarrow \tilde{E}$ a continuous map with values in some Polish space $\tilde{E}$, then the $\mc{P}(\tilde{E})$-valued random variables $H_\# L^{N,\eps}$ converge in law to $H_\# L$. This follows from the continuity of the map $\nu \mapsto H_\#\nu$, which in turn follows from the continuity of $H$.
\item Let $g$ be in $C_b(E)$, then the real-valued random variables $L^{N,\eps}(g)$ converge in law to $L(g)$. Similarly to the previous point, this follows from the continuity of the map $\nu \mapsto \nu(g)$.
\item The probability measures $Q^{N,\eps,e}$ converge weakly to $Q^e$: indeed, for every $g$ in $C_b(E)$, $\E[L^{N,\eps}(g)]$ converge to $\E[L(g)]$.
\item For any Borel set $B$ of $E$, it holds $L^{N,\eps}(B)=1$ $P$-a.s.~if and only if $Q^{N,\eps,e}(B)=1$; similarly for $L$ and $Q^{\eps}$.
\end{itemize}
\end{remark}

Now we show that, for a.e.~$L$, $W$ (the first component in $E$) is a Brownian motion under $L$ and $X$ is a generalized solution to the McKean-Vlasov equation, with the right initial condition, under $L$. Roughly speaking, we would like to pass to the limit (as $N\to \infty$ and $\eps\to 0$) in equation \eqref{eq:particle_equiv} and get the McKean-Vlasov SDE \eqref{eq:meanfield_SDE_equiv}. The proof is in two parts. In the first part, we prove that the expectation of $X_t$ under $L$ is $q(t)$, this implies the first line in \eqref{eq:meanfield_SDE_equiv}; we also prove that the law of the initial condition $X_0$ and $W$ is $\text{Law}(X_0)\otimes \text{Wiener measure}$. In the second part we idenfity the reflection term and prove its properties, getting the equalities for $\bar{Z}$ and $\bar{k}$ in \eqref{eq:meanfield_SDE_equiv}.

\begin{lemma}\label{lemma:lambda}
It holds $P$-a.s.:
\begin{itemize}
\item For every $t$,
\begin{align*}
\E^L[X_t] = q(t).
\end{align*}
\item Under the measure $L$, the random path
\begin{align*}
t\mapsto X_t-X_0-\sigma W_t-Z_t
\end{align*}
is actually $L$-a.s.~deterministic; that is, the law of this path under $L$ is a Dirac delta.
\end{itemize}
\end{lemma}

\begin{proof}
For the first point, we start fixing $t$. The function on $E$ defined by $(W,X,Z)\mapsto X_t$ is continuous and bounded (as $X$ takes values in $[0,1]$). Therefore, by Remark \ref{rmk:conv_random_meas}, the random variables $\E^{L^{N,\eps}}[X_t]$ converge in law to $\E^L[X_t]$. On the other hand, equation \eqref{SDE_single_reg} gives, for every $(N,\eps)$, $P$-a.s.,
\begin{align*}
\E^{L^{N,\eps}}[X_t] = \frac{1}{N}\sum^N_{i=1} X^{i,N,\eps} =q(t).
\end{align*}
Hence the law of $\E^L[X_t]$ is $\delta_{q(t)}$, that is $\E^L[X_t]=q(t)$ on a $P$-full measure set $\Omega_t$, which may depend on $t$. To make the exceptional set independent on $t$, we note that, by dominated convergence theorem, $t\mapsto \E^L[X_t]$ is continuous for every $\omega$ and that $q$ is also continuous by assumption, hence we have the equality for every $t$ in the full-measure set $\Omega' = \cap_{s\in\mathbb{Q}\cap[0,T]} \Omega_s$. The proof of the first point is complete.

For the second point, we have to prove that $(X-X_0-\sigma W-Z)_\#L$ is a Dirac delta $P$-a.s.. By Remark \ref{rmk:conv_random_meas}, the $\mc{P}(C_t)$-valued random variables $(X-X_0-\sigma W-Z)_\#L^{N,\eps}$ converge in law to $(X-X_0-\sigma W-Z)_\#L$. On the other hand, equation \eqref{SDE_single_reg} gives, for every $(N,\eps)$, $P$-a.s.: for every $i=1,\ldots N$,
\begin{align*}
&(X-X_0-\sigma W-Z)_\#\delta_{(W^i,X^{i,N,\eps},-\int^\cdot_0\mu^\eps(X^{i,N,\eps}_r)\rmd r -k^{i,N,\eps})}\\
&= \delta_{X^{i,N,\eps}-X^{i,N,\eps}_0-\sigma W^i_t +\int^\cdot_0\mu^\eps(X^{i,N,\eps}_r)\rmd r +k^{i,N,\eps}}\\
&=\delta_{q(t) -q(0) +\frac{1}{N}\sum^N_{j=1} [\int^\cdot_0\mu^\eps(X^{j,N,\eps}_r)\rmd r -\sigma W^j_t +k^{j,N,\eps} ]}=:\delta_{\gamma^{N,\eps}},
\end{align*}
note that $\gamma^{N,\eps}$ is independent of $i$. Averaging over $i$, we get
\begin{align*}
(X-X_0-\sigma W-Z)_\#L^{N,\eps} = \delta_{\gamma^{N,\eps}},
\end{align*}
in particular $(X-X_0-\sigma W-Z)_\#L^{N,\eps}$ is concentrated on the subset $\{\delta_{\gamma}\mid \gamma\in C_t\}$ of $\mc{P}(C_t)$. We claim that $\{\delta_{\gamma}\mid \gamma\in C_t\}$ is a closed set in $\mc{P}(C_t)$. Hence, since $(X-X_0-\sigma W-Z)_\#L^{N,\eps}$ converges in law to $(X-X_0-W-Z)_\#L$, also $(X-X_0-\sigma W-Z)_\#L$ is concentrated on $\{\delta_{\gamma}\mid \gamma\in C_t\}$, that is the law of $X-X_0-\sigma W-Z$ under $L$ is a Dirac delta.

It remains to prove the above claim. If $\delta_{\gamma^n}$ converge to a measure $\nu$, then, by tightness of $\delta_{\gamma^n}$, there exists a compact set $K$ in $C_t$ such that $\delta_{\gamma^n}(K)>1/2$ and so $\gamma^n$ belong to $K$, for every $n$. Therefore there exists a subsequence $\gamma^{n_k}$ converging to some element $\gamma$ in $K$, hence $\delta_{\gamma^{n_k}}$ converge to $\delta_\gamma$ and so $\nu=\delta_\gamma$ belongs to $\{\delta_{\gamma}\mid \gamma\in C_t\}$, which is then closed. The proof of the second point is complete.
\end{proof}

\begin{lemma}\label{lemma:BM_lambda}
It holds $P$-a.s.: under $L$, the $C_t\times \mr$-valued random variable $(W,X_0)$ has law $P^W\otimes\text{Law}(X_0)$, where $P^W$ is the Wiener measure.
\end{lemma}

\begin{proof}
The map from $E$ to $C_t\times \mr$ defined by $(W,X,Z)\mapsto (W,X_0)$ is continuous. Therefore, by Remark \ref{rmk:conv_random_meas}, the random empirical measures
\begin{align*}
(W,X_0)_\# L^{N,\eps} = \frac{1}{N}\sum^N_{i=1} \delta_{(W^i,X^{i,N}_0)}
\end{align*}
converge in law to $(W,X_0)_\# L$. On the other hand, the above random measures converge in law to $P^W\otimes \text{Law}(X_0)$ (see e.g. \cite[Lemma 29]{CDFM2020}). Hence the law of $\text{Law}^L(W,X_0)$ is $\delta_{P^W\otimes \text{Law}(X_0)}$, that is $\text{Law}^L(W,X_0)=P^W\otimes \text{Law}(X_0)$ $P$-a.s.. The proof is complete.
\end{proof}

Next we define the process $k$ by
\begin{align}
k_t = k(X,Z)_t = -\int^t_0 1_{X_r\notin ]0,1[} \rmd Z_r,\label{def_k}
\end{align}
if $Z$ is a $BV$ path on $[0,T]$, $k_t=0$ otherwise. We call $|k|$ the total variation process associated with $k$.

\begin{lemma}\label{lemma:BV}
It holds $P$-a.s.:
\begin{itemize}
\item The processes $Z$, $\int^\cdot_0\mu(X_r)\rmd r$, $k$ have $BV$ trajectories $L$-a.e.~and their $BV$ norms are $p$-integrable with respect to $L$, for any $1\le p<\infty$.
\item It holds $L$-a.e.: for every $t\ge 0$,
\begin{align}
Z_t + k_t = \int^t_0 1_{X_r\in ]0,1[} \rmd Z_r = -\int^t_0 \mu(X_r)1_{X_r\in ]0,1[} \rmd r.\label{eq:drift_term}
\end{align}
\item  It holds $L$-a.e.: the process $k$ satisfies the condition
\begin{align}
\rmd|k|_t = 1_{X_t\in \{0,1\}}\rmd|k|_t,\ \ \rmd k_t = n(X_t)\rmd|k|_t.\label{eq:boundary_term}
\end{align}
\end{itemize}
\end{lemma}

\begin{proof}
For all statements but the $p$-integrability of the $BV$ norms, by Remark \ref{rmk:conv_random_meas}, it is enough to prove these statements $Q^e$-a.e.~(recall $Q^e$ is the average of $L$) instead of $L$-a.e.~(provided we work with Borel sets/properties, as the proof will do); it is also enough to prove $p$-integrability of the $BV$ norms of $Z$, $\int^\cdot_0\mu(X_r)\rmd r$, $k$ with respect to $Q^e$. Again by Remark \ref{rmk:conv_random_meas}, the measures $Q^{N,\eps,e}$ converge in law to $Q^e$, hence we can work with $Q^{N,\eps,e}$ and $Q^e$ only.

\textbf{$BV$ property of $Z$ and $k$}. By Lemma \ref{unif_BV}, we have
\begin{align}
\E^{Q^{N,\eps,e}} \|Z\|_{BV}^p \le \E \frac{1}{N}\sum^N_{i=1} \left(\int^T_0 |\mu^\eps(X^{i,N,\eps}_r)|\rmd r +|k^{i,N,\eps}|_T\right)^p \le C\label{unif_BV_Z}
\end{align}
for some constant $C$ independent of $\eps$ and $N$. Now the $BV$ norm is lower semi-continuous in $C_t$, since it can be written as
\begin{align*}
\|\gamma\|_{BV} = \sup_\pi \sum_{[t_i,t_{i+1}[\in\pi}|\gamma(t_{i+1})-\gamma(t_i)|,
\end{align*}
the $\sup$ being over all partitions $\pi$ of $[0,T]$. Therefore it holds
\begin{align*}
\E^{Q^e} \|Z\|_{BV}^p \le C,
\end{align*}
in particular $Z$ and so $k$ have $BV$ paths $Q^e$-a.s..

\textbf{Support property of $|k|$}. By definition, $k$ is concentrated on $\{t\in [0,T]\mid X_t\in \{0,1\}\}$, which is a closed set, hence also its total variation process $|k|$ is concentrated on this set and we conclude that
\begin{align}
\rmd|k|_t = 1_{X_t\in \{0,1\}}\rmd|k|_t.\label{eq:supp_k}
\end{align}

\textbf{$BV$ property of $\int^\cdot_0\mu(X_r)\rmd r$}. Since $\mu^\eps\ge \mu^\delta$ for $\eps<\delta$, by monotone convergence theorem we have
\begin{align*}
&\E^{Q^e} \left(\int^T_0 |\mu(X_r)|\rmd r\right)^p = \sup_\delta \E^{Q^e} \left(\int^T_0 |\mu^\delta(X_r)|\rmd r\right)^p\\
&= \sup_\delta \lim_{N,\eps} \left(\E^{Q^{N,\eps,e}} \int^T_0 |\mu^\delta(X_r)|\rmd r\right)^p\\
&= \sup_\delta \lim_{N,\eps} \E\frac{1}{N}\sum^N_{i=1} \left(\int^T_0 |\mu^\delta(X^{N,\eps,i}_r)|\rmd r\right)^p\\
&\le \liminf_{N,\eps} \E\frac{1}{N}\sum^N_{i=1} \left(\int^T_0 |\mu^\eps(X^{N,\eps,i}_r)|\rmd r\right)^p <\infty,
\end{align*}
in particular also $\int^\cdot_0\mu(X_r)\rmd r$ has $BV$ trajectories, with $Q^e$-integrable $BV$ norm.

\textbf{Representation formulae for $Z+k$ and $k$}. We take $a>0$, $\delta>0$, $\varphi:[0,1]\rightarrow\mr$ $C^1$ function with support in $[\delta,1-\delta]$, $\tilde{n};[0,1]\rightarrow\mr$ a continuous extension of the outer normal $n$ with support on $]\delta,1-\delta[^c$ and with $\tilde{n}\ge 0$ on $[1-\delta,1]$ and $\tilde{n}\le 0$ on $[0,\delta]$, $g,h:[0,T]\rightarrow \mr$ continuous with $g$ non-negative. We consider the set
\begin{align*}
&A=A_{a,\varphi,\tilde{n},h,g} \\
&= \{(W,X,Z) \in E \mid \|Z\|_{BV}\le a,\ \int^T_0 h(r)\varphi(X_r) \rmd Z_r = -\int^T_0 h(r)\varphi(X_r)\mu(X_r) \rmd r,\ \int^T_0 g(r)\tilde{n}(X_r) \rmd Z_r \le 0 \}.
\end{align*}

\begin{lemma}
The set $A$ is closed in $E$.
\end{lemma}

\begin{proof}
Let $(W^n,X^n,Z^n)$ be a sequence in $A$ converging to $(W,X,Z)$ uniformly. Since the $BV$ norm of $Z^n$ is bounded by $a$ for every $n$, up to taking a subsequence we can assume that $dZ^n$ converges weakly-* to a measure $\nu$ with total variation $\|\nu\|_{TV}\le a$. Passing to the limit in the chain rule for $Z$, we find that, for every $\psi$ in $C^\infty([0,T])$,
\begin{align*}
\int^T_0 \psi\rmd \nu = \psi(T)Z_T - \psi(0)Z_0 -\int^T_0 \psi' Z\rmd r.
\end{align*}
Hence $\nu$ is the distributional derivative of $Z$, which therefore satisfies $\|Z\|_{BV}\le a$.

Concerning the stability of the conditions involving $\mu$ and $\tilde{n}$, note that $h(r)\varphi(X^n_r)\rightarrow h(r)\varphi(X_r)$ uniformly and also $h(r)\varphi(X^n_r)\mu(X^n_r)\rightarrow h(r)\varphi(X_r)\mu(X_r)$ uniformly (since $\mu$ is $C^1$ on $[\delta,1-\delta]$). This fact and the weak-* convergence of $Z^n$ implies that
\begin{align*}
&\int_{[0,T]} h(r)\varphi(X_r) \rmd Z_r = \lim_n \int^T_0 h(r)\varphi(X^n_r) \rmd Z^n_r\\
&= - \lim_n \int^T_0 h(r)\varphi(X^n_r)\mu(X^n_r) \rmd r = - \int^T_0 h(r)\varphi(X_r)\mu(X_r) \rmd r.
\end{align*}
Reasoning similarly for $\tilde{n}$, we find
\begin{align*}
\int^T_0 g(r)\tilde{n}(X_r) \rmd Z_r = \lim_n \int^T_0 g(r)\tilde{n}(X^n_r) \rmd Z^n_r \le 0.
\end{align*}
This proves that $(W,X,Z)$ is in $A$. Hence $A$ is closed.
\end{proof}

Now the equation for $X^{\eps,N}$ and Condition \ref{assumptions_mu}-(iii) imply, for $\delta<\rho$: for $\eps<\delta$, under $Q^{\eps,N,e}$ it holds a.s.
\begin{align*}
&\int^T_0 h(r)\varphi(X_r) \rmd Z_r = -\int^T_0 h(r)\varphi(X_r)\mu(X_r) \rmd r,\\
&\int^T_0 g(r)\tilde{n}(X_r) \rmd Z_r = -\int^T_0 g(r)\tilde{n}(X_r)(\mu^\eps(X_r)\rmd r +\rmd k_r) \le 0.
\end{align*}
Moreover the uniform bound \eqref{unif_BV_Z} implies
\begin{align*}
Q^{\eps,N,e}\{\|Z\|_{BV}> a\}\le \frac{1}{a}\E^{Q^{\eps,N,e}} \|Z\|_{BV} \le C/a.
\end{align*}
Therefore, for $\delta<\rho$, for any $a$, $Q^{\eps,N,e}(A) \ge 1-C/a$. Since $A$ is closed, we conclude that $Q^e(A) \ge 1-C/a$. Hence $Q^e$ is concentrated on the set
\begin{align*}
&B_{\varphi,\tilde{n},h,g}=\\
&\{(W,X,Z) \in \Omega \mid \|Z\|_{BV}<\infty,\ \int^T_0 h(r)\varphi(X_r) \rmd Z_r = -\int^T_0 h(r)\varphi(X_r)\mu(X_r) \rmd r,\ \int^T_0 g(r)\tilde{n}(X_r) \rmd Z_r \le 0 \},
\end{align*}
for every $\varphi$, $\td{n}$, $h$, $g$ as above. Now we take: $\varphi=\varphi^m$ tending pointwise to $1_{]0,1[}$ and uniformly bounded in $m$; $\tilde{n}=\tilde{n}^m$ tending pointwise to $n(x)1_{\{0,1\}}$ and uniformly bounded in $m$; $h$ in acountable dense set $D$ in $C_t$, $g$ in $D^+$ countable dense set $D^+$ in $\{g\in C_t\mid g\ge 0\}$. Therefore we have
\begin{align}
Q^e\text{ is concentrated on }\td{B} = \cap_{m,h\in D,g\in D^+} B_{\varphi^m,\tilde{n}^m,h,g} \cap \{(W,X,Z)\mid \int^\cdot_0 \mu(X)\rmd r \in BV\}.\label{eq:Qe_supp}
\end{align}

\begin{lemma}\label{lem:tdB}
For every $(W,X,Z)$ in $\td{B}$, it holds:
\begin{align}
&\int^t_0 1_{X_r\in ]0,1[} \rmd Z_r = -\int^t_0 \mu(X_r)1_{X_r\in ]0,1[} \rmd r\ \ \forall t,\label{repr_mu}\\
&\rmd k_t = n(X_t)\rmd|k|_t.\label{repr_k}
\end{align}
\end{lemma}

\begin{proof}
For every $(W,X,Z)$, for every fixed $h$ in $D$ and $g$ in $D^+$, the $BV$ property of $Z$ (and so of $k$) and of $\int^\cdot_0\mu(X_r)\rmd r$ implies, via dominated convergence theorem,
\begin{align*}
&\int^T_0 h(r)\varphi^m(X_r) \rmd Z_r \rightarrow \int^T_0 h(r)1_{X\in]0,1[} \rmd Z_r,\ \ \int^T_0 h(r)\varphi^m(X_r)\mu(X_r) \rmd r \rightarrow \int^T_0 h(r)1_{X\in]0,1[}\mu(X_r) \rmd r,\\
&\int^T_0 g(r)\tilde{n}^m(X_r) \rmd Z_r \rightarrow -\int^T_0 g(r)n(X_r) \rmd k_r.
\end{align*}
Therefore, if $(W,X,Z)$ is in $\td{B}$, passing to the limit in $m$ in the definition of $B_{\varphi,\tilde{n},h,g}$ we get
\begin{align*}
&\int^T_0 h(r)1_{X\in]0,1[} \rmd Z_r = -\int^T_0 h(r)\mu(X_r)1_{X\in]0,1[} \rmd r,\\
&-\int^T_0 g(r)n(X_r) \rmd k_r \le 0.
\end{align*}
for all $h$ in $D$, $g$ in $D^+$. By the density of $D$ and $D^+$ we obtain \eqref{repr_mu} and that $n(X_r)dk_r\ge 0$, which together with \eqref{eq:supp_k} implies \eqref{repr_k}.
\end{proof}

Thanks to \eqref{eq:Qe_supp} and Lemma \ref{lem:tdB}, we conclude that, for $Q^e$-a.e. $(W,X,Z)$, the representation formulae \eqref{repr_mu} and \eqref{repr_k} hold. Therefore \eqref{eq:drift_term} and \eqref{eq:boundary_term} hold. The proof of Lemma \ref{lemma:BV} is complete.
\end{proof}

\begin{remark}\label{rmk:tech_2}
Only in the above proof we use Condition \ref{assumptions_mu}-(iii). If $\sigma\neq0$, we expect that, by a suitable version of Girsanov theorem on domains, the time spent by $X$ on the boundary has zero Lebesgue measure, $Q$-a.s.. Morally this should allow to remove or relax the condition \ref{assumptions_mu}-(iii).
\end{remark}

We are ready to prove:

\begin{proposition}
It holds $P$-a.e.: under $L$, $(X,k)$ is a generalized solution to the McKean-Vlasov problem \eqref{meanfield_SDE_single} starting from $X_0$, with initial distribution $\text{Law}(X_0)$ (more precisely, $(E,\mc{B}(E),W,X,k,L)$ is a weak generalized solution with initial distribution $\text{Law}(X_0)$).
\end{proposition}

\begin{proof}
By Lemma \ref{lemma:BM_lambda}, $P$-a.s.,~$W$ is a Brownian motion under $L$ and $X_0$ is independent of $W$. As a consequence of Lemma \ref{lemma:lambda} it holds, $P$-a.s.,~under $L$: for every $t$,
\begin{align}
X_t = X_0 +Z_t +\sigma W_t +q(t)-q(0) -\E^L Z_t,
\end{align}
where we have used $\E^L W_t=0$. By Lemma \ref{lemma:BV}, it holds, $P$-a.s.,~under $L$: $\int^\cdot_0\mu(X_r)1_{X_r\in]0,1[}\rmd r$ and $k$ are in $BV$ with integrable $BV$ norms and, for every $t$,
\begin{align}
Z_t = -\int^t_0\mu(X_r)1_{X_r\in]0,1[}\rmd r -k_t = -\int^t_0\mu(X_r)\rmd r -k_t,\label{eq:Z_Xk}
\end{align}
where $k_t$ satisfies \eqref{eq:boundary_term} and where we have used that $\mu(0)=\mu(1)=0$. Therefore $(X,k)$ satisfies \eqref{eq:meanfield_SDE_equiv} and so it is a generalized solution.
\end{proof}

We deduce, via Yamada-Watanabe, the existence of a strong solution to \eqref{meanfield_SDE_single}, that is the existence part of Theorem \ref{thm:one}, as well as uniqueness in law:

\begin{corollary}\label{cor:limit_pt}
It holds $P$-a.s.: under $L$, $(X,k)$ is a strong solution to the SDE \eqref{meanfield_SDE_single} and the law of $X$ under $L$ coincide with the unique law $\text{Law}(\bar{X})$ of any solution to \eqref{meanfield_SDE_single} starting from $\text{Law}(\bar{X}_0)=\text{Law}(X_0)$.
\end{corollary}

\begin{proof}
We have $P$-a.s.: the couple $(X,k)$ is a weak generalized solution under $L$, hence it is a strong solution, via Proposition \ref{YamWat_exist}. Proposition \ref{YamWat_uniq} gives uniqueness in law for the $X$ component.
\end{proof}

Finally we arrive at the convergence result, that is Theorem \ref{thm:convergence}:

\begin{corollary}
The family $(\frac{1}{N}\sum^N_{i=1}\delta_{X^{\eps,N,i}} = X_\#L^{\eps,N})_{\eps,N}$ of random probability measures on $C([0,T];[0,1])$ converges in probability, as $\eps\rightarrow 0$ and $N\rightarrow\infty$, to the law of the McKean-Vlasov solution $\bar{X}$ (starting from $\text{Law}(X_0)$).
\end{corollary}

\begin{proof}
Since the limit $\text{Law}(\bar{X})$ is deterministic (and $\mc{P}(E)$ is a metric space), it is enough to prove convergence in law. Since $\mc{P}(\mc{P}(E))$ is a metric space and the family $(\text{Law}(X_\#L^{\eps,N})_{\eps,N})_{\eps,N}$ is relatively compact (that is tight), it is enough to prove that every limit point of $(\text{Law}(X_\#L^{\eps,N})_{\eps,N})_{\eps,N}$ is actually $\delta_{\text{Law}(\bar{X})}$. This is an immediate consequence of Corollary \ref{cor:limit_pt}. The proof is complete.
\end{proof}

\subsection{Pathwise analysis}

This subsection is dedicated to the proof of Proposition \ref{pro: convergence rate}; we assume in this Subsection the conditions of Proposition \ref{pro: convergence rate}. We use a pathwise approach developed e.g. in \cite{CDFM2020}, we explain first briefly the core idea behind it. Let $(\Omega, \mathcal{A}, \mathbb{P})$ be a probablity space and $W: \Omega \to C([0,T], \mathbb{R})$ a random variable on this space. Note that at this point we do not impose that $W$ is a Brownian Motion.
Consider the SDE
\begin{align}
\begin{aligned}
\label{eq: boundary sde}
&\rmd X_t = [ -\mu( X _t) + \mathbb{E}[\mu ( X _t )] ] \rmd t + \dot{q}_t \rmd t + \sigma \rmd W_t - \rmd k_t - \sigma \mathbb{E}[\rmd W_t] + \mathbb{E}[\rmd k_t]\\
& X\in C([0,T];[0,1]),\  k\in C([0,T];\mr),\\
&\rmd| k| = 1_{X_t\in \{0,1\}}\rmd| k|,\ \ \rmd  k = n( X_t)\rmd| k|.
\end{aligned}
\end{align}
If we endow the probability space with a (right-continuous complete) filtration $(\mathcal{F}_t)_{t\geq 0}$ and assume that $W$  is a Brownian Motion with respect to this filtration, clearly equation \eqref{eq: boundary sde} is exactly the McKean-Vlasov equation \eqref{meanfield_SDE_single}.

On the other hand, let $(X^{(N)}, k^{(N)})$ be the solution of the interacting particle system \eqref{SDE_single_reg}. Let $\omega \in \Omega$ be fixed. On a suitable discrete prabability space endowed with the point counting measure the process $(X^{(N)}, k^{(N)})(\omega) = (X^{i,N}(\omega), k^{i,N}(\omega))_{i=1,\dots,N}$ is a random variable in the variable $i$ and as such a solution to equation \eqref{eq: boundary sde}. The mean with respect to the point counting measure is exactly the empirical average.

This is the main idea behind the proof of the Lemma \ref{lem: continuity in wasserstein 2}. First we recall the definition of Wasserstein distance.
\begin{definition}
	\label{def: wasserstein}
	Let $(E,d)$ be a polish space. Let $\mathcal{P}_2(E)$ be the space of probability measures on $E$ with finite second moment. The $2$-Wasserstein distance on $\mathcal{P}(E)$ is defined as 
	\begin{equation*}
		\mathcal{W}_{2,E}(\mu,\nu) := \inf\left\{
		\left(
		\int_{E\times E} d(x,y)^2 m(dx,dy)
		\right)^{\frac{1}{2}}
		\mid
		m \mbox{ coupling of }\mu,\nu
		\right\}
		\qquad
		\mu,\nu \in \mathcal{P}_2(E).
	\end{equation*}
\end{definition}

From now on, we work under the assumptions of Proposition \ref{pro: convergence rate}.

\begin{lemma}
	\label{lem: continuity in wasserstein 2}
	Let $(\bar{X}, \bar{k})$ be the solution to equation
	\eqref{meanfield_SDE_single} with initial condition $\bar{X}_0$ with law $\nu_0$. Let $( X^{ ( N ) } , k^{ ( N ) } )$ be a solution to the interacting particle system \eqref{SDE_single_reg} with initial condition $X^{(N)} = (X^{1,N}_0, \dots, X^{N,N}_0)$. Assume that $(\bar{X}_0, X^{(N)}_0)$ is independent on the noise $W^{(N)} = (W^1, \dots, W^N)$.
	For every $t\in [0,T]$, we have $\mathbb{P} - a.s.$,
	\begin{align*}
	\mathcal{W}_{2,\mathbb{R}} ( \text{Law} ( \bar{X}_t), \frac1N \sum_{ i = 1 }^{ N } \delta_{X ^ { i , N } _ t} )^2  
	\leq & C \left ( 1 + \mathbb{E}\left[\left(\int_{0}^{T}\rmd|\bar{k}|_s\right)^2\right] + \frac1N \sum_{ i = 1 }^{ N }  \left(\int_{0}^{T}\rmd|k^{ i , N}|_s\right)^2 \right ) \\
	& \cdot \mathcal{W}_{2, C([0,T], \mathbb{R} )} ( \text{Law} ( W ), \frac1N \sum_{ i = 1 }^{ N } \delta_{W ^ { i } } )^2
	+ \mathcal{W}_{2,[0,1]} ( \nu^0 , \frac1N \sum_{ i = 1 }^{ N } \delta_{X ^ { i,N }_0 } )^2.
	\end{align*}
\end{lemma}

\begin{proof}
	For simplicity of notation, we take $\sigma=1$ (the argument is the same for general $\sigma\in \mr$). Call $\nu := \text{Law}(W)$ the Wiener measure on $C_t=C([0,T], [0,1])$.
	For a fixed $\omega \in \Omega$, we consider the empirical measure $L^{N,\omega} := \frac{1}{N}\sum_{i=1}^N \delta_{(W^i(\omega), X^{i,N}_0(\omega))}$ as a law on $E = ( C_t \times [0,1], \mathcal{B}( C_t ) \times \mathcal{B}([0,1]) )$. 
	Let $P^{\omega} \in \mathcal{P}(E\times E)$ be any coupling of $\nu \otimes \nu_0$ and $L^{N,\omega}$. It is easy to verify that $P^{\omega}$ can be seen as a measure on $\Omega^\omega := (E \times \{(W^1(\omega), X^{1,N}_0(\omega)), \dots , (W^N(\omega),X^{N,N}_0(\omega))\})$, endowed with the product $\sigma$-algebra $ \mathcal{A}^\omega := \mathcal{B}( C_t ) \times 2^{  \{(W^1(\omega), X^{1,N}_0(\omega)), \dots , (W^N(\omega),X^{N,N}_0(\omega))\}} $. Indeed, for every Borel bounded test function $\varphi : E \times E \to \mathbb{R}$,
	\begin{align*}
	P^\omega ( \varphi )
	& = \int_{ E \times E } \varphi( x, y ) P^\omega(\rmd x, y) L^{N,\omega}(\rmd y)\\
	& = \frac{1}{N} \sum_{ i = 1 } ^ { N } \int_{ E} \varphi( x , (W^i(\omega), X^{i,N}_0(\omega)) P^\omega(\rmd x, \rmd (W^i(\omega),X^{i,N}_0(\omega)))\\
	& = \int_{ \Omega^\omega } \varphi( x , y) P^\omega(\rmd x, y) L^{N,\omega}(\rmd y).
	\end{align*}
	
	On the space $(\Omega^\omega, \mathcal{A}^\omega , P^\omega)$ we define the projections $(\Pi^1,\Pi^1_0) \sim \nu \otimes \nu_0$ and $(\Pi^2,\Pi^2_0) \sim L^{N, \omega}$ on the first and second marginal space, respectively (in particular, $\Pi^1 \sim \nu$ and $\Pi^2 \sim \frac{1}{N}\sum_{i=1}^{N}\delta_{W^i(\omega)}$). 
	Since the law of $\Pi ^ 1$ is the Wiener measure $\nu$, we have that $\Pi ^ 1$ is a Brownian motion, and if we plug it as the driver of equation \eqref{eq: boundary sde} we obtain a strong unique solution $(\bar X, \bar k)$ thanks to Theorem \ref{thm:one}.
	
	Let $( X ^ { ( N ) }, k ^ { ( N ) } )$ be the solution of equation \eqref{SDE_single_reg} given by Proposition \ref{lem:wellpos_particle}. There exists a set of full measure $\Omega_0 \subset \Omega$ such that for every $\omega \in \Omega_0$ and every $1 \leq i \leq N$, $(X^{ i, N} ( \omega ), k^{ i , N } ( \omega ) ) $ satisfies equation \eqref{SDE_single_reg}. Defining $(\tilde X, \tilde k)(W^i(\omega), X^{i,N}_0(\omega) ) := (X^{ i , N } _t , k ^ { i , N } _t ) ( \omega )$, we have that, for every $t\in [0,T]$, $\mathbb{E}_{ P^\omega } [ \tilde k_t ] = \frac1N \sum_{ j = 1 }^{ N } k ^ { j , N} _t ( \omega )$ and
	\begin{equation*}
	\rmd \tilde X_t
	=  ( \mu ( \tilde X _t) - \mathbb{E}_{P^\omega}[ \mu ( \tilde X_t) ] ) \rmd t + \rmd q_t +  \rmd \Pi ^ 2 _t -  \rmd\tilde k_t - \rmd\mathbb{E}_{ P^\omega } [ \Pi^2_t ] + \rmd\mathbb{E}_{ P^\omega } [ \tilde k _t ],
	\quad \mbox{on } \Omega_0.
	\end{equation*}
	We define $b( \bar X_t) : = \mu ( \bar X_t ) - \mathbb{E}_{P^\omega} [\mu ( \bar X_t )]$ and we estimate the following
	\begin{align*}
	\frac12 & \rmd ( \bar X_t - \tilde X_t - (\Pi ^1 _t - \Pi ^2 _t ) + \mathbb{E}_{P^\omega}[\Pi ^1 _t - \Pi ^2 _t] 
	+ \int_{0}^t [ b ( \bar X_s ) - b ( \tilde X _s) ] \rmd s ) ^2\\
	& =  - ( \bar X_t - \tilde X_t ) \rmd \bar k_t + ( \bar X_t - \tilde X_t ) \rmd \tilde k_t \\
	&\quad + ( \bar X_t - \tilde X_t ) \rmd \mathbb{E}_{P^\omega}[ \bar k_t ] - ( \bar X_t - \tilde X_t ) \rmd \mathbb{E}_{P^\omega}[ \bar k_t ] 
	+ ( \Pi^1 _t - \Pi^2 _t) \rmd \bar k_t - ( \Pi^1 _t - \Pi^2 _t) \rmd \tilde k_t \\
	&\quad - ( \Pi^1 _t - \Pi^2 _t) \rmd \mathbb{E}_{P^\omega} [ \bar k_t ] +( \Pi^1 _t - \Pi^2 _t) \rmd \mathbb{E}_{P^\omega} [ \tilde k_t ] 
	- \mathbb{E}_{P^\omega} [ \Pi^1 _t - \Pi^2 _t] \rmd \bar k_t + \mathbb{E}_{P^\omega} [ \Pi^1 _t - \Pi^2 _t] \rmd \tilde k_t \\
	&\quad + \mathbb{E}_{P^\omega} [ \Pi^1 _t - \Pi^2 _t] \rmd \mathbb{E}_{P^\omega} [ \bar k_t ] - \mathbb{E}_{P^\omega} [ \Pi^1 _t - \Pi^2 _t] \rmd \mathbb{E}_{P^\omega} [\tilde k_t ]
	+ \left[ \int_{0}^t [ b ( \bar X_s ) - b ( \tilde X _s) ] \rmd s \right] \rmd(\mathbb{E}_{P^\omega} [ \bar k_t ] - \mathbb{E}_{P^\omega} [\tilde k_t ]).
	\end{align*}
	The first and second term on the right-hand side are always negative by the conditions on the boundaries. If we take expectation under $P^\omega$ on both sides, we have that the third and fourth term on the right-hand side vanish, because $\mathbb{E}_{P^\omega} [\bar X_t ] = \mathbb{E}_{P^\omega}[ \tilde X_t ] = q_t$. Similarly, the expectation of the last term vanishes. Hence, we have that
	\begin{align*}
	\mathbb{E}_{P^\omega} & [ \; | \bar X_t - \tilde X_t - (\Pi^1_0 - \Pi^2_0)- (\Pi ^1 _t - \Pi ^2 _t ) + \mathbb{E}_{P^\omega}[\Pi ^1 _t - \Pi ^2 _t]
	+ \int_{0}^t [ b ( \bar X_t ) - b ( \tilde X _s) ] \rmd s   | ^2] \\
	\leq & 2 \mathbb{E}_{P^\omega} [ \sup_{ t \in [ 0 , T ] } | \Pi^1_t - \Pi^2_t|^2] \left( \mathbb{E}_{P^\omega}\left[\left(\int_{0}^{T} \rmd|\bar k|_s \right)^2\right] + \mathbb{E}_{P^\omega}\left[\left(\int_{0}^{T} \rmd|\tilde k|_s\right)^2\right] \right)\\
	&+ \mathbb{E}_{P^\omega} [ | \Pi^1_0 - \Pi^2_0|^2]
	\end{align*}
	The proof is concluded by first using Gronwall's lemma and then choosing $P^\omega = P^{\omega}_{W} \otimes P^{\omega}_0$, where $P^{\omega}_W$ (resp. $P^{\omega}_0$) is the optimal coupling in $\mathcal{W}_{2,C_t} ( \nu , \frac1N \sum_{ i = 1 }^{ N } \delta_{W ^ { i } } )$ (resp. $\mathcal{W}_{2,[0,1]} ( \nu_0 , \frac1N \sum_{ i = 1 }^{ N } \delta_{X ^ { i,N }_t } )$)
\end{proof}

Thanks to the previous proposition, it is immediate to derive the convergence of the particle system to the McKean-Vlasov equation, provided that we have convergence at time $0$ and a bound on the second moment of $k^N$.

\begin{proof}[Proof of Proposition \ref{pro: convergence rate}]
	By Lemma \ref{lem: continuity in wasserstein 2} and using H\"older inquality, we have
	\begin{align*}
	\mathbb{E}\left[ \sup_{ t \in [ 0 , T ] } \mathcal{W}_{2,[0,1]} ( \text{Law} ( \bar X_t), \frac1N \sum_{ i = 1 }^{ N } \delta_{X ^ { i , N } _t} )  \right]
	\leq  &C \left ( 1+ \mathbb{E} \left[\left(\int_{0}^{T} \rmd|\bar k|_s \right)^2\right] +  \frac1N \sum_{ i = 1 }^{ N }  \mathbb{E}\left[\left(\int_{0}^{T} \rmd|k^{ i , N}|_s \right)^2 \right] \right ) ^{\frac{1}{2}}\\
	& \cdot \mathbb{E} \left[ 
	\mathcal{W}_{2,C([0,T], \mathbb{R} )} ( \text{Law} ( W ), \frac1N \sum_{ i = 1 }^{ N } \delta_{W ^ { i } } )^2 \right]^{\frac{1}{2} }\\
	& + C\mathbb{E} \left[  \mathcal{W}_{2,[0,1]} ( \text{Law}(\bar{X}_0) , \frac1N \sum_{ i = 1 }^{ N } \delta_{X ^ { i,N }_0 } ) \right].
	\end{align*}
	The first term on the right-hand side is uniformly bounded in $N$ thanks to Lemma \ref{BV_sum} and the $2$-integrability of $\|\bar{k}\|_{BV}$ in Lemma \ref{lemma:BV}. The empirical measure of independent random variables distributed as the the Wiener measure convegres in Wasserstein metric to the Wiener measure as $O(1/\sqrt{\log(N)})$, see \cite{boissard2014mean}. 
	The Wasserstein distance of the intial conditions converges faster. Remember $X^{i,N}_0 = Y^i + \sum_{j=1}^N\delta_{Y^j} + q(0)$, where $(Y^i)_{i\in\mathbb{N}}$ is a family of independent and identically distributed random variables. We see that the speed of convergnce of $\frac{1}{N}\sum_{i=1}^N \delta_{X^{i,N}_0}$ is the same as the the speed of convergence of $\sum_{i=1}^N\delta_{Y^j}$, which is $1/\sqrt{N}$, see \cite{fournier2015rate}. For a fixed $\omega \in \Omega$, take an optimal coupling $m = m(\omega)$ between $\text{Law}(\bar{X}_0)$ and $\frac{1}{N}\sum_{i=1}^{N} \delta_{Y^i}$, we have that $(x, y-\mathbb{E}_{m}[y] -\mathbb{E}_{m}[x])_{\#}m(dx,dy)$ is a coupling between $\text{Law}(\bar{X}_0)$ and $\frac{1}{N}\sum_{i=1}^{N} \delta_{X_0^i}$. We can compute
	\begin{align*}
		\mathcal{W}_{2,[0,1]} ( \text{Law}(\bar{X}_0) , \frac1N \sum_{ i = 1 }^{ N } \delta_{X ^ { i,N }_0 } )^2
		\leq &\mathbb{E}_{m}\left[
		|X - Y + \mathbb{E}_{m}Y -\mathbb{E}_mX|^2
		\right]
		= \operatorname{Var}_{m}(X-Y)\\
		\leq &\mathbb{E}_{m}|X-Y|^2
		= \mathcal{W}_{2,[0,1]} ( \text{Law}(\bar{X}_0)  , \frac1N \sum_{ i = 1 }^{ N } \delta_{Y^ { i,} } )^2.
	\end{align*}
	Taking the square roots and the expectation under $\mathbb{P}$ concludes the proof.

\end{proof}

%

\section{Appendix: Proof of Proposition \ref{lem:wellpos_particle}}\label{app:B}

The system \eqref{SDE_single_reg} can be seen as an SDE on the moving domain $H_t\cap [0,1]^N$, where
\begin{align*}
H_t = \{x\in \mr^N \mid \frac{1}{N}\sum_{i=1}^N x^i =q(t)\},
\end{align*}
with normal boundary conditions. Indeed, formally, for each $i=1,\ldots N$ and $m=0,1$, on the boundary $x^i=m$, the direction of reflection $(-1)^m(e_i-N^{-1}(1,\ldots 1))$ ($e^i$ being the $i$-th vector of the canonical basis) is orthogonal to the face $H_t\cap \{x\mid x^i=m\}$. Here we use this fact to show well-posedness of the system \eqref{SDE_single_reg}.

We introduce some notation. In the following, we fix $N$ and omit the superscripts $N$ and $\eps$ in the notation. We call $H=\{x\in\mr^N \mid \frac{1}{N}\sum_{i=1}^N x^i =0\}$, $\mathrm{1} = (1,1,\ldots 1)\in \mr^N$, $\Pi:\mr^N\to\mr^N$ the projector on $H$, that is $\Pi x= x-N^{-1}(x\cdot \mathrm{1})\mathrm{1}$. We take $A:H\to \mr^{N-1}$ a linear isometry and we call $D_t = A\Pi(H_t\cap [0,1]^N)$. For $i=1,\ldots N$, $m=0,1$, we call $\partial_{i,m}[0,1]^N = \{x\in [0,1]^N\mid x^i=m\}$, $\partial_{i,m}D_t = A\Pi(H_t\cap \partial_{i,m}[0,1]^N)$, $\gamma_{i,m} = (-1)^m (e_i-N^{-1}\mathrm{1})$ the direction of reflection of \eqref{SDE_single_reg} on the face $\partial_{i,m}[0,1]^N$ and $\nu_{i,m} = A\gamma_{i,m}$. For $x$ in $\partial[0,1]^N = \cup_{i,m} \partial_{i,m}[0,1]^N$, we call
\begin{align*}
\Gamma(x) = \{\sum_{i,m} c_{i,m}\gamma_{i,k}1_{x\in \partial_{i,k}[0,1]^N} \mid c_{i,m}\ge 0 \,\,\forall i=1,\ldots N,m=0,1\}.
\end{align*}
Similarly, for $y$ in $\partial D_t = \cup_{i,m} \partial_{i,m}D_t$, we call
\begin{align*}
N_t(y) = \{\sum_{i,m} c_{i,m}\nu_{i,k}1_{y\in \partial_{i,k}D_t} \mid c_{i,m}\ge 0 \,\,\forall i=1,\ldots N,m=0,1\},
\end{align*}
note that $N_t(y) = A\Pi(x)$ if $x$ is in $\partial_{i,m}[0,1]^N$.

We consider the following SDE on $D_t$:
\begin{align}
\begin{aligned}\label{eq:SDE_moving_plane}
&\rmd Y_t = A\Pi b(t,A^{-1}Y_t +q(t)\mathrm{1}) \rmd t +\Pi\rmd W_t +\rmd h_t,\\
&Y_t \in D_t \,\,\forall t,\quad P\text{-a.s.},\\
&\rmd|h|_t = 1_{Y_t\in \partial D_t}\rmd|h|_t,\ \ \rmd h_t = \nu_t\rmd|h|_t,\ \ \nu_t \in N_t(Y_t),
\end{aligned}
\end{align}
where $(Y,h)$ is the solution, $W$ is an $N$-dimensional Brownian motion with respect to a (complete, right-continuous) filtration $(\mc{F}_t)_t$ and $b$ is the drift of the system \eqref{SDE_single_reg}. This is an SDE on a moving domain $D_t$ with reflection at the boundary. As we will see, the SDE \eqref{eq:SDE_moving_plane} is, up to the isometry $A$, the system \eqref{SDE_single_reg}.

\begin{lemma}\label{lem:exist_SDE_moving_plane}
Under Condition \ref{assumptions_q} on $q$ and the Lipschitz continuity of $\mu^\eps$, given a probability space $(\Omega,\mc{A},P)$ and a Brownian motion $W$ on a (complete right-continuous) filtration $(\mc{F}_t)_t$, there exists a unique (strong) solution to the SDE \eqref{eq:SDE_moving_plane}.
\end{lemma}

\begin{proof}
The existence and uniqueness result is a consequence of \cite[Theorem 1.7]{nystrom2015remarks} for SDEs on moving domains with reflecting boundaries, provided that the assumptions of that theorem hold. We focus on two key assumptions, namely: a) the fact that $N_t(y)$ is the cone of inward normal vectors of $D_t$ at $y$, for every $t$ and every $y\in \partial D_t$; b) relation (1.16) in \cite{nystrom2015remarks}. The other assumptions of \cite[Theorem 1.7]{nystrom2015remarks} are easy to verify.

Concerning assumption a), we observe that, for each $i=1,\ldots N$, $m=0,1$, the vector $\gamma_{i,m}$ is the inward normal, in the $N-1$-dimensional convex polyhedron $H_t\cap [0,1]^N$, of the corresponding face $H_t\cap \partial_{i,m}[0,1]^N$: indeed $\gamma_{i,m}$ belongs to $H$ and, for every $v$ in $H\cap \partial_{i,m}[0,1]^N$, we have $\gamma_{i,m}\cdot v=0$. Since $A$ is an isometry, the vector $\nu_{i,m}=A\gamma_{i,m}$ is the inward normal, in the convex polyhedron $D_t$, of the corresponding face $\partial_{i,m}D_t = A\Pi(H_t\cap \partial_{i,m}[0,1]^N)$. Now $N_t(y)$ is the convex cone generated by $\nu_{i,m}$, with $i,m$ such that $y\in \partial_{i,m}D_t$. Hence $N_t(y)$ is the convex cone of inward normal vectors (in the sense of \cite[Definition 2.2]{nystrom2015remarks}), see e.g. formula (4.23) in \cite{costantini1992skorohod}.

Assumption b) reads as follows. Define
\begin{align}
a_{s,z}(\rho,\eta) = \max_{u\in\mr^{N-1},|u|=1}\, \min_{s\le t\le t+\eta}\, \min_{y\in \partial D_t,|y-z|\le \rho}\, \min_{\nu \in N_t(y), |\nu|=1}\, (\nu \cdot u).\label{eq:hp_b_def}
\end{align}
Condition (1.16) in \cite{nystrom2015remarks} reads
\begin{align}
\lim_{\eta\to 0}\lim_{\rho \to 0} \inf_{s\in [0,T]} \inf_{z\in \partial D_s} a_{s,z}(\rho,\eta) =a >0.\label{eq:hp_b}
\end{align}
In order to show this condition, we take
\begin{align*}
&I=I_{s,z}(\rho,\eta) = \cup_{s\le t\le t+\eta} \, \cup_{y\in \partial D_t,|y-z|\le \rho} \, \{(i,m)\mid y\in \partial_{i,m}D_t\},\\
&u_{s,z}(\rho,\eta) = c_I \sum_{(i,m)\in I_{s,z}(\rho,\eta)} \nu_{i,m},
\end{align*}
where $c_I>0$ is a positive constant such that $|u_{s,z}(\rho,\eta)|=1$; note that $\min_I c_I = c_N>0$. We also note that, for suitable $\rho_0>0$, $\eta_0>0$ (independent of $s$ and $z$), for every $\rho<\rho_0$ and $\eta<\eta_0$, for every $s$ and $z$, for every $j=1,\ldots N$, at most one element between $(j,0)$ and $(j,1)$ belongs to $I_{s,z}(\rho,\eta)$. Moreover, since the average of $y$ is in $[\xi,1-\xi]$ (for all $y\in \partial D_t$ for all $t$), $y^i$ cannot be all $0$, nor they can be all $1$, hence $I_{s,z}$ cannot be $\{(1,0),(2,0),\ldots (N,0)\}$ or $\{(1,1),(2,1),\ldots (N,1)\}$. As a consequence,
\begin{align}
\text{if }(i,m)\in I,\text{ then there exist at most } N-2 \text{ indices } j\neq i \text{ with }(j,m)\in I.\label{eq:nb_indices}
\end{align}
We compute the scalar products among $\nu_{i,m}$, using the isometry property of $A$:
\begin{align*}
&|\nu_{i,m}|^2 = |\gamma_{i,m}|^2 = 1-N^{-1},\\
&\nu_{i,m}\cdot \nu_{j,m} = \gamma_{i,m}\cdot \gamma_{j,m} = -N^{-1} \quad \text{for }i\neq j,\\
&\nu_{i,m}\cdot \nu_{j,n} = \gamma_{i,m}\cdot \gamma_{j,n} = N^{-1} \quad \text{for }i\neq j,m\neq n.
\end{align*}
We call $\hat{\nu}_{i,m} = (1-N^{-1})^{-1/2}\nu_{i,m}$. For $\eta<\eta_0$ and $\rho<\rho_0$, we get by \eqref{eq:nb_indices}, for every $(i,m)$ in $I=I_{s,z}(\rho,\eta)$,
\begin{align*}
\hat{\nu}_{i,m}\cdot u_I &= c_I (1-N^{-1})^{-1/2} \left( 1-N^{-1} -\sum_{(j,m)\in I,j\neq i} N^{-1} +\sum_{(j,n)\in I,j\neq i,m\neq n}N^{-1} \right)\\
&= c_I (1-N^{-1})^{-1/2} (1-N^{-1} -(N-2)N^{-1})\\
&\ge c_N N^{-1} (1-N^{-1})^{-1/2}.
\end{align*}
Now, for every $s,z$, for every $t\in [s,s+\eta]$ and $y\in \partial D_t$ with $|y-z|\le \rho$, $N_t(y)$ is contained in the convex cone generated by $\hat{\nu}_{i,m}$, $(i,m)\in I_{s,z}(\eta,\rho)$. Therefore, for $\eta<\eta_0$ and $\rho<\rho_0$, for every $s$ and $z$, we have
\begin{align*}
\nu \cdot u_I \ge c_N N^{-1} (1-N^{-1})^{-1/2},\quad \text{for every } \nu \text{ as in \eqref{eq:hp_b_def}}.
\end{align*}
and so $a_{s,z}(\rho,\eta) \ge c_N N^{-1} (1-N^{-1})^{-1/2}>0$, in particular \eqref{eq:hp_b} holds. The proof is complete.
\end{proof}

Now we show that the SDE \eqref{eq:SDE_moving_plane} is equivalent to the system \eqref{SDE_single_reg}. We introduce some notation. We take a Borel map
\begin{align*}
G:\{(t,y,v)\mid t\in [0,T], y\in \partial D_t, v\in N_t(y)\} \rightarrow [0,+\infty)^{N\times 2},\quad (t,x,v)\mapsto (c_{i,m})_{i=1,\ldots N,m=0,1},
\end{align*}
such that $c_{i,m}=0$ if $y$ does not belong to $\partial_{i,m}D_t$, and
\begin{align*}
v= \sum_{(i,m),y\in \partial_{i,m} D_t} c_{i,m}\nu_{i,m}.
\end{align*}
[Note that this map $G$ exists but is not uniquely determined: indeed, if $y$ belongs to $\partial_{i,0} D_t \cup \partial_{i,1}D_t$ for each $i$ (that is, $y=A\Pi x$ for some $x$ with $x^i\in \{0,1\}$ for each $i$), then $\nu_{i,m}$ are not linearly independent.] For a solution $(Y,h)$ to \eqref{eq:SDE_moving_plane}, with $\rmd h= \nu \rmd |h|$, we call
\begin{align*}
&X^{Y,h}_t = A^{-1}Y_t + q(t)\mathrm{1},\\
&k^{Y,h}_t = \int_0^t \sum_{(i,m)} G_{i,m}(r,Y_r,\nu_r) n(X^i_r)e_i 1_{X^i_r = m} \rmd |h|_r,
\end{align*}
recall that $n(m)=-(-1)^m$ is the outward normal of $[0,1]$ in $m=0,1$.

\begin{lemma}\label{lem:link_SDE}
Assume that $(Y,h)$ is a solution to \eqref{eq:SDE_moving_plane}. Then $(X^{Y,h},k^{Y,h})$ is a solution to the system \eqref{SDE_single_reg}.
\end{lemma}

\begin{proof}
Let $(Y,h)$ be a solution to \eqref{eq:SDE_moving_plane}, take $(X,k)=(X^{Y,h},k^{Y,h})$. By the definition of $X^{Y,h}$ and $k^{Y,h}$, $P$-a.s. $X$ is has continuous paths with values in $[0,1]$ and $k$ has continuous paths, and, for each $i$, $|k^i|$ is concentrated on $\{t\mid X_i\in \{0,1\}\}$ and has direction $n(X^i)$. Hence the second and third lines of \eqref{SDE_single_reg} are satisfied. We have
\begin{align*}
\rmd h_t &= \sum_{(i,m)} G_{i,m}(t,Y_t,\nu_t) 1_{Y_t\in \partial_{i,m} D_t} \nu_{i,m} \rmd |h|_t\\
&= \sum_{(i,m)} G_{i,m}(t,Y_t,\nu_t) 1_{X^i_t=m} A(-1)^m (e_i -N^{-1}\mathrm{1}) \rmd |h|_t\\
&= -\sum_{(i,m)} G_{i,m}(t,Y_t,\nu_t) 1_{X^i_t=m} n(X^i_t) Ae_i \rmd |h|_t +N^{-1}\sum_{(i,m)} G_{i,m}(t,Y_t,\nu_t) 1_{X^i_t=m} n(X^i_t) A\mathrm{1} \rmd |h|_t\\
&= -A\rmd k_t +N^{-1}A(\rmd k \cdot \mathrm{1}) \mathrm{1}= A(-\rmd k_t +N^{-1}\sum_i \rmd k^i_t \mathrm{1}).
\end{align*}
Hence, applying the transformation $X_t =A^{-1}Y_t +q(t)\mathrm{1}$ to the first line of \eqref{eq:SDE_moving_plane}, we obtain the first line of \eqref{SDE_single_reg}. Therefore $(X^{Y,h},k^{Y,h})$ satisfies \eqref{SDE_single_reg}. The proof is complete.
\end{proof}

By Lemmas \ref{lem:exist_SDE_moving_plane} and \ref{lem:link_SDE}, we get existence of a solution to \eqref{SDE_single_reg}.

\begin{remark}
We expect also the converse of Lemma \ref{lem:link_SDE} to hold, namely, if $(X,k)$ solves \eqref{SDE_single_reg}, then $(Y=A\Pi X,h=-A\Pi k)$ solves \eqref{eq:SDE_moving_plane}. In particular, from this converse we would get uniqueness for \eqref{SDE_single_reg}. However showing the third line of \eqref{eq:SDE_moving_plane} is not immediate, hence we do not follow this strategy.
\end{remark}

We conclude the proof of Proposition \ref{lem:wellpos_particle} by showing uniqueness for \eqref{SDE_single_reg}:

\begin{lemma}
Strong uniqueness (in $X$) holds for the SDE \eqref{SDE_single_reg}.
\end{lemma}

\begin{proof}
The proof follows the line of Proposition \ref{prop:uniq_McKVla}, replacing the expectation with the empirical average. Let $(X,k^X)$, $(Y,k^Y)$ two solutions to \eqref{SDE_single_reg} with the same initial condition $X_0=Y_0$. In this proof we call $\rmd K^X = \frac{1}{N}\sum^N_{j=1}[\mu(X^j)\rmd t -\rmd W^j + \rmd k^{X,j}]$ and similarly for $K^Y$. By It\^o formula for continuous semimartingales, we have, for every $i=1,\ldots N$,
\begin{align*}
&\rmd|X^i-Y^i|^2\\
&= 2(X^i-Y^i)(-\mu(X^i)+\mu(Y^i))\rmd t +2(X^i-Y^i)\rmd K^X -2(X^i-Y^i)\rmd K^Y\\
&\ \ -2(X^i-Y^i)\rmd k^{X,i} +2(X^i-Y^i)\rmd k^{Y,i}.
\end{align*}
The one-side Lipschitz condition of $\mu$ implies
\begin{align*}
(X^i-Y^i)(-\mu(X^i)+\mu(Y^i))\le c|X^i-Y^i|^2.
\end{align*}
and the orientation of $k$ (as the outward normal) implies
\begin{align*}
-\int^t_0(X^i-Y^i)\rmd k^{X,i} \le 0
\end{align*}
and similarly for $(X^i-Y^i)\rmd k^{Y,i}$. For the addends with $K$, we average over $i$ and use that $K$ does not depend on $i$ and that $\frac{1}{N}\sum_i X^i_t = \frac{1}{N}\sum_i Y^i_t =q(t)$: we obtain
\begin{align*}
\frac{1}{N}\sum_i (X^i-Y^i)\rmd K^X =0
\end{align*}
and similarly for $(X^i-Y^i)\rmd K^Y$. Putting all together, we get
\begin{align*}
\frac{1}{N}\sum_i|X^i_t-Y^i_t|^2 \le C\int^t_0 \frac{1}{N}\sum_i|X^i_r-Y^i_r|^2\rmd r.
\end{align*}
We conclude by Gronwall inequality that $\frac{1}{N}\sum_i|X^i_t-Y^i_t|^2 =0$, that is $X=Y$. The proof is complete.
\end{proof}

\bibliographystyle{alpha}
\bibliography{bibliography}

\newcommand{\etalchar}[1]{$^{#1}$}
\def\cprime{$'$}
\begin{thebibliography}{BCCdRH20}

\bibitem[AGS08]{ambrosio2008gradient}
Luigi Ambrosio, Nicola Gigli, and Giuseppe Savar{\'e}.
\newblock {\em Gradient flows: in metric spaces and in the space of probability
  measures}.
\newblock Springer Science \& Business Media, 2008.

\bibitem[Aid16]{Aid2016}
Shigeki Aida.
\newblock Rough differential equations containing path-dependent bounded
  variation terms, 2016.
\newblock arXiv:1608.03083.

\bibitem[Bar20]{Bar2020}
Clayton~L. Barnes.
\newblock Hydrodynamic limit and propagation of chaos for {B}rownian particles
  reflecting from a {N}ewtonian barrier.
\newblock {\em Ann. Appl. Probab.}, 30(4):1582--1613, 2020.

\bibitem[BCCdRH20]{BriCarChaHu2020}
Philippe Briand, Pierre Cardaliaguet, Paul-\'{E}ric Chaudru~de Raynal, and Ying
  Hu.
\newblock Forward and backward stochastic differential equations with normal
  constraints in law.
\newblock {\em Stochastic Process. Appl.}, 130(12):7021--7097, 2020.

\bibitem[BCD20]{BaiCatDel2020}
Isma\"{e}l Bailleul, R\'{e}mi Catellier, and Fran\c{c}ois Delarue.
\newblock Solving mean field rough differential equations.
\newblock {\em Electron. J. Probab.}, 25:Paper No. 21, 51, 2020.

\bibitem[BCdRGL20]{BriChaGuiLab2020}
Philippe Briand, Paul-\'{E}ric Chaudru~de Raynal, Arnaud Guillin, and
  C\'{e}line Labart.
\newblock Particles systems and numerical schemes for mean reflected stochastic
  differential equations.
\newblock {\em Ann. Appl. Probab.}, 30(4):1884--1909, 2020.

\bibitem[BEH18]{BriEliHu2018}
Philippe Briand, Romuald Elie, and Ying Hu.
\newblock B{SDE}s with mean reflection.
\newblock {\em Ann. Appl. Probab.}, 28(1):482--510, 2018.

\bibitem[BJ11]{BosJab2011}
Mireille Bossy and Jean-Fran\c{c}ois Jabir.
\newblock On confined {M}c{K}ean {L}angevin processes satisfying the mean
  no-permeability boundary condition.
\newblock {\em Stochastic Process. Appl.}, 121(12):2751--2775, 2011.

\bibitem[BJ15]{BosJab2015}
Mireille Bossy and Jean-Fran\c{c}ois Jabir.
\newblock Lagrangian stochastic models with specular boundary condition.
\newblock {\em J. Funct. Anal.}, 268(6):1309--1381, 2015.

\bibitem[BJ18]{BosJab2018}
Mireille Bossy and Jean-Fran\c{c}ois Jabir.
\newblock Particle approximation for {L}agrangian stochastic models with
  specular boundary condition.
\newblock {\em Electron. Commun. Probab.}, 23:Paper No. 15, 14, 2018.

\bibitem[BLG14]{boissard2014mean}
Emmanuel Boissard and Thibaut Le~Gouic.
\newblock On the mean speed of convergence of empirical and occupation measures
  in wasserstein distance.
\newblock In {\em Annales de l'IHP Probabilit{\'e}s et statistiques},
  volume~50, pages 539--563, 2014.

\bibitem[CCP11]{CacCarPer2011}
Mar\'{\i}a~J. C\'{a}ceres, Jos\'{e}~A. Carrillo, and Beno\^{\i}t Perthame.
\newblock Analysis of nonlinear noisy integrate \& fire neuron models: blow-up
  and steady states.
\newblock {\em J. Math. Neurosci.}, 1:Art. 7, 33, 2011.

\bibitem[CD18]{CarDelBookI}
Ren\'{e} Carmona and Fran\c{c}ois Delarue.
\newblock {\em Probabilistic theory of mean field games with applications.
  {I}}, volume~83 of {\em Probability Theory and Stochastic Modelling}.
\newblock Springer, Cham, 2018.
\newblock Mean field FBSDEs, control, and games.

\bibitem[CDFM20]{CDFM2020}
Michele Coghi, Jean-Dominique Deuschel, Peter~K. Friz, and Mario Maurelli.
\newblock Pathwise {M}c{K}ean-{V}lasov theory with additive noise.
\newblock {\em Ann. Appl. Probab.}, 30(5):2355--2392, 2020.

\bibitem[CL15]{MR3299600}
Thomas Cass and Terry Lyons.
\newblock Evolving communities with individual preferences.
\newblock {\em Proc. Lond. Math. Soc. (3)}, 110(1):83--107, 2015.

\bibitem[Cos92]{costantini1992skorohod}
Cristina Costantini.
\newblock The {S}korohod oblique reflection problem in domains with corners and
  application to stochastic differential equations.
\newblock {\em Probab. Theory Related Fields}, 91(1):43--70, 1992.

\bibitem[DEH19]{DjeEliHam2019}
Boualem Djehiche, Romuald Elie, and Said Hamadène.
\newblock Mean-field reflected backward stochastic differential equations,
  2019.
\newblock arXiv:1911.06079.

\bibitem[DGH11]{DreGulHer2011}
Wolfgang Dreyer, Clemens Guhlke, and Michael Herrmann.
\newblock Hysteresis and phase transition in many-particle storage systems.
\newblock {\em Contin. Mech. Thermodyn.}, 23(3):211--231, 2011.

\bibitem[DGHT19]{DeyGubHofTin2019}
Aur\'{e}lien Deya, Massimiliano Gubinelli, Martina Hofmanov\'{a}, and Samy
  Tindel.
\newblock One-dimensional reflected rough differential equations.
\newblock {\em Stochastic Process. Appl.}, 129(9):3261--3281, 2019.

\bibitem[DHM{\etalchar{+}}15]{DHMRW2015}
Wolfgang Dreyer, Robert Huth, Alexander Mielke, Joachim Rehberg, and Michael
  Winkler.
\newblock Global existence for a nonlocal and nonlinear {F}okker-{P}lanck
  equation.
\newblock {\em Z. Angew. Math. Phys.}, 66(2):293--315, 2015.

\bibitem[DI91]{MR1110990}
Paul Dupuis and Hitoshi Ishii.
\newblock On {L}ipschitz continuity of the solution mapping to the {S}korokhod
  problem, with applications.
\newblock {\em Stochastics Stochastics Rep.}, 35(1):31--62, 1991.

\bibitem[DIRT15]{DelIngRubTan2015}
Fran\c{c}ois Delarue, James Inglis, Sylvain Rubenthaler, and Etienne Tanr\'{e}.
\newblock Global solvability of a networked integrate-and-fire model of
  {M}c{K}ean-{V}lasov type.
\newblock {\em Ann. Appl. Probab.}, 25(4):2096--2133, 2015.

\bibitem[FG15]{fournier2015rate}
Nicolas Fournier and Arnaud Guillin.
\newblock On the rate of convergence in wasserstein distance of the empirical
  measure.
\newblock {\em Probability Theory and Related Fields}, 162(3):707--738, 2015.

\bibitem[FR13]{FerRov2013}
Marco Ferrante and Carles Rovira.
\newblock Stochastic differential equations with non-negativity constraints
  driven by fractional {B}rownian motion.
\newblock {\em J. Evol. Equ.}, 13(3):617--632, 2013.

\bibitem[GGM{\etalchar{+}}18]{GGMFD2018}
Clemens Guhlke, Paul Gajewski, Mario Maurelli, Peter~K. Friz, and Wolfgang
  Dreyer.
\newblock Stochastic many-particle model for {LFP} electrodes.
\newblock {\em Contin. Mech. Thermodyn.}, 30(3):593--628, 2018.

\bibitem[HLSj19]{HamLedSoj2019}
Ben Hambly, Sean Ledger, and Andreas S\o~jmark.
\newblock A {M}c{K}ean-{V}lasov equation with positive feedback and blow-ups.
\newblock {\em Ann. Appl. Probab.}, 29(4):2338--2373, 2019.

\bibitem[Jab17]{Jab2017}
Jean-Francois Jabir.
\newblock Diffusion processes with weak constraint through penalization
  approximation, 2017.
\newblock arXiv:1704.01505.

\bibitem[Kol07]{Kol2007}
Vassili~N. Kolokoltsov.
\newblock Nonlinear {M}arkov semigroups and interacting {L}\'{e}vy type
  processes.
\newblock {\em J. Stat. Phys.}, 126(3):585--642, 2007.

\bibitem[Li14]{Li2014}
Juan Li.
\newblock Reflected mean-field backward stochastic differential equations.
  {A}pproximation and associated nonlinear {PDE}s.
\newblock {\em J. Math. Anal. Appl.}, 413(1):47--68, 2014.

\bibitem[Lio08]{Lio2008}
Pierre-Louis Lions.
\newblock Lectures at coll\`{e}ge de france, 2008.

\bibitem[LS84]{LioSzn1984}
Pierre-Louis Lions and Alain-Sol Sznitman.
\newblock Stochastic differential equations with reflecting boundary
  conditions.
\newblock {\em Comm. Pure Appl. Math.}, 37(4):511--537, 1984.

\bibitem[M\'96]{Mel1996}
Sylvie M\'{e}l\'{e}ard.
\newblock Asymptotic behaviour of some interacting particle systems;
  {M}c{K}ean-{V}lasov and {B}oltzmann models.
\newblock In {\em Probabilistic models for nonlinear partial differential
  equations ({M}ontecatini {T}erme, 1995)}, volume 1627 of {\em Lecture Notes
  in Math.}, pages 42--95. Springer, Berlin, 1996.

\bibitem[Men83]{Men1983}
Jos\'{e}-Luis Menaldi.
\newblock Stochastic variational inequality for reflected diffusion.
\newblock {\em Indiana Univ. Math. J.}, 32(5):733--744, 1983.

\bibitem[NO15]{nystrom2015remarks}
Kaj Nystr\"{o}m and Thomas \"{O}nskog.
\newblock Remarks on the {S}korohod problem and reflected {L}\'{e}vy driven
  {SDE}s in time-dependent domains.
\newblock {\em Stochastics}, 87(5):747--765, 2015.

\bibitem[RY99]{revuz1999}
Daniel Revuz and Marc Yor.
\newblock {\em Continuous martingales and {B}rownian motion}, volume 293 of
  {\em Grundlehren der Mathematischen Wissenschaften [Fundamental Principles of
  Mathematical Sciences]}.
\newblock Springer-Verlag, Berlin, third edition, 1999.

\bibitem[Son07]{Son2007}
Halil~Mete Soner.
\newblock Stochastic representations for nonlinear parabolic {PDE}s.
\newblock In {\em Handbook of differential equations: evolutionary equations.
  {V}ol. {III}}, Handb. Differ. Equ., pages 477--526. Elsevier/North-Holland,
  Amsterdam, 2007.

\bibitem[Szn84]{Szn1984}
Alain-Sol Sznitman.
\newblock Nonlinear reflecting diffusion process, and the propagation of chaos
  and fluctuations associated.
\newblock {\em J. Funct. Anal.}, 56(3):311--336, 1984.

\bibitem[Szn91]{sznitman1991topics}
Alain-Sol Sznitman.
\newblock Topics in propagation of chaos.
\newblock In {\em Ecole d'{\'e}t{\'e} de probabilit{\'e}s de Saint-Flour
  XIX—1989}, pages 165--251. Springer, 1991.

\bibitem[Tan79]{Tan1979}
Hiroshi Tanaka.
\newblock Stochastic differential equations with reflecting boundary condition
  in convex regions.
\newblock {\em Hiroshima Math. J.}, 9(1):163--177, 1979.

\bibitem[Tan84]{MR780770}
Hiroshi Tanaka.
\newblock Limit theorems for certain diffusion processes with interaction.
\newblock In {\em Stochastic analysis ({K}atata/{K}yoto, 1982)}, volume~32 of
  {\em North-Holland Math. Library}, pages 469--488. North-Holland, Amsterdam,
  1984.

\end{thebibliography}

\end{document}